\documentclass[onefignum,onetabnum]{siamart171218}
\UseRawInputEncoding

\usepackage{graphicx,color}
\usepackage{hyperref}
\usepackage{dsfont}
\usepackage{float}
\usepackage{capt-of}
\usepackage{mathtools}
\usepackage{amssymb}
\usepackage{stmaryrd}
\usepackage{lipsum}
\usepackage{amsfonts}
\usepackage{amsmath}

\usepackage{stackrel}

\usepackage{graphicx}
\usepackage{epstopdf}
\usepackage{algorithmic}
\ifpdf
  \DeclareGraphicsExtensions{.eps,.pdf,.png,.jpg}
\else
  \DeclareGraphicsExtensions{.eps}
\fi

\newcommand{\eqalign}[1]{\begin{aligned}#1\end{aligned}}

\newcommand{\R}{\mathbb{R}}
\newcommand{\N}{\mathbb{N}}

\newsiamremark{remark}{Remark}
\newsiamremark{hypothesis}{Hypothesis}
\crefname{hypothesis}{Hypothesis}{Hypotheses}
\newsiamthm{claim}{Claim}

\headers{Uniqueness Lipschitz stability  and reconstruction }{H.  Meftahi}

\title{Uniqueness, Lipschitz stability and  reconstruction   for the  inverse  optical tomography problem \thanks{Submitted to the editors DATE.}}

\author{ Houcine   Meftahi\thanks{Department of Mathematics, ENIT of Tunisia
  \email{houcine.meftahi@enit.utm.tn}.}}

\usepackage{amsopn}

\def\norm#1{\hspace{0.2ex} \|#1\| \hspace{0.2ex}} 

\begin{document}
\maketitle

\begin{abstract}
In this paper, we consider the inverse problem of recovering a diffusion  $\sigma$ and  absorption coefficients $q$ in steady-state optical tomography  
problem from the Neumann-to-Dirichlet map.  We first   prove a  Global uniqueness and Lipschitz stability estimate for  the absorption  parameter  
provided that the diffusion   $\sigma$  is known and show how to quantify the  Lipschitz stability constant for a given setting. Then,  we  prove a  Lipschitz stability result  for  simultaneous recovery  of  $\sigma$ and $q$. In both cases 
the   parameters belong to a known finite subspace with a priori known bounds. The proofs rely on a  monotonicity result combined with the techniques of localized potentials. 
To numerically solve the inverse problem, we propose a   Kohn-Vogelius-type cost functional over a class of admissible parameters subject to two boundary value problems. The
reformulation of the minimization problem via the Neumann-to-Dirichlet operator allows us to obtain the optimality conditions by using the Fr\'echet differentiability of
this operator and its inverse. The reconstruction is then performed by means of an iterative algorithm based on a quasi-Newton method. Finally, we  illustrate some 
 numerical   results.
\end{abstract} 
\begin{keywords}
 Optical tomography, Inverse problem, Uniqueness, Lipschitz stability, Monotonicity,  Localized potentials. 
 \end{keywords}
\begin{AMS}
 78A46, 65J22, 65M32, 35R30
 \end{AMS}
\section{Introduction}
In this paper, we consider the inverse problem of recovering  the  parameters $\sigma(x)$ and $q(x)$  in the elliptic partial differential  equation
\begin{equation}
\label{equation_intro}
-\nabla\cdot(\sigma \nabla u)+qu = 0 \textrm{ in }  \Omega, 
\end{equation}
from the  knowledge of all possible Cauchy data on the boundary $\partial\Omega$,
$ \sigma\partial_\nu u|_{\partial\Omega}, u|_{\partial\Omega}$.\\
Problem (\ref{equation_intro}) can be  viewed as   steady-state  diffusion optical tomography, where
light propagation is modeled by a diffusion approximation and the excitation frequency is set to zero.
Here $u$  represents the density of photons, $\sigma$ the diffuse coefficient  and $q$ the optical absorption. This  problem arises    in medical imaging   and  in geophysics, for example, in reflection seismology assuming a description in terms of time-harmonic scalar waves.
 For a full description of optical tomography,  we refer the reader to the topical reviews of Arridge \cite{arridge1999optical} and Gibson, Hebden and Arridge \cite{gibson2005recent}.
 
 Although it is common practice in optical  tomography to use the Robin-to-Robin
map to describe the boundary measurements (see \cite{arridge1999optical,heino2002estimation}), the Neumann-to-Dirichlet  map will be employed here instead. This is justified by the fact that in optical tomography, prescribing the Neumnann to-Dirichlet  map, is equivalent to prescribing the Robin-to-Robin
boundary map  as long as there are no additional unknown coefficients in the Robin conditions (see for instance \cite{harrach2009uniqueness}).
\\
 The paper  is split into three parts.  Part one is on proving uniqueness and Lipschitz stability  of the 
the absorption coefficient $q$ provided that the diffusion coefficient $\sigma$ is known.  Part two is on proving Lipschitz 
stability  of $\sigma$  and $q$  simultaneously.  Part three  deals with  the reconstruction of $\sigma$ and $q$ based on minimizing a Kohn-Vogelius  type  functional.

 The inverse problem of recovering   $q$ from the knowledge
of  the Dirichlet-to-Neumann map was first introduced (in a slightly different setting) by Calder\'on in \cite{calderon1980inverse}.  The uniqueness issue was treated by Sylvester and Uhlmann in \cite{sylvester1987global}.
For more recent result on uniqueness,  we refer the reader to \cite{uhlmann2009electrical}. 
By virtue of the work of Alessandrini  \cite{alessandrini1988stable}  it is known that both problems of recovering $\sigma$ or $q$ (in suitable regularity scales) enjoy logarithmic stability estimates   under mild a priori assumptions on the data. As
shown by Mandache \cite{mandache2001exponential}, this log-type estimate is optimal. Thus for arbitrary potentials $q$, Lipschitz stability cannot hold. As  discovered   in \cite{alessandrini2005lipschitz}, considering potentials or conductivities in certain finite-dimensional spaces provides improvements in terms of stability. Under certain  assumptions, the authors prove Lipschitz stability estimates.  Their argument  relies  on a combination of singular  solutions  and unique continuation estimates. This idea has been extended to more  complex equations and systems  (see for instance  \cite{beretta2011lipschitz,alessandrini2017lipschitz,gaburro2015lipschitz,alessandrini2018lipschitz,beretta2013lipschitz,
alberti2019calderon}).

As a key novelty  in this article, we  present a different approach based on the monotonicity and  the techniques of localized potentials  instead of  combining singular solutions with unique continuation results as previously
done in the literature. Following analogous results in electrical impedance
tomography and elasticity  \cite{harrach2019uniqueness,harrach2019global,eberle2019lipschitz},
here we will study the question whether the coefficient $q$  can be uniquely and stably reconstructed.  More precisely,  we show that $q$    is uniquely determined and depends upon the Neumaun-to-Dircihlet map  of (\ref{equation_intro})  in a Lipschitz way as long as  $\rm{supp}(\it q) \Subset\Omega$ and $\sigma$ is known.  Moreover,
we  quantify  the Lipschitz constant   for a given setting   by solving a  finite number of well-posed PDEs
which may be important to quantify the  noise robustness in practical applications. To our best knowledge,  this result  of 
quantitative Lipschitz stability  is new for the problem  under consideration.

As mentioned in  \cite{arridge1998nonuniqueness},   the  inverse problem of simultaneous  reconstruction of  $\sigma$ and $q$  is in general  not uniquely solvable, i.e., it is not possible to uniquely determine both $\sigma$ and $q$ from boundary data of $u$ provided that $\sigma$ and $q$ are smooth. The reason is that a diffusion coefficient can be transformed into an absorption coefficient by setting
\[
 v:=\sqrt{\sigma}u,
\]
which transforms equation (\ref{equation_intro})  into
\[
-\Delta v+ cv=0, \quad  c=\frac{\Delta\sqrt{\sigma}}{\sqrt{\sigma}}+\frac{q}{\sigma}.
\]
If $\sigma =1$ in a neighborhood of $\partial\Omega$,  then the boundary values remain unchanged. Hence,
boundary measurements can only contain information about $c$, from which one cannot
extract $\sigma$ and $q$. 
Despite this negative theoretical result, a prominent result by Harrach \cite{harrach2009uniqueness}  demonstrates 
that uniqueness holds for piecewise constant diffusion and piecewise analytic absorption coefficients.
The author  proves that under this condition both
parameters are simultaneously uniquely determined by knowledge of all possible pairs
of Neumann and Dirichlet boundary values $\sigma\partial_\nu\vert_S,  u\vert_S$,  of solutions  $u$ of (\ref{equation_intro}), and $S$  is  a non-empty  subset of $\partial\Omega$

In this paper, we go a step further    and we prove a  Lipschitz stability  for the inverse problem of recovering $q$   and 
$\sigma$ simultaneously.  The proof  relies  on a monotonicity estimates combined with the techniques of  localized potentials.
To the author's knowledge the  Lipschitz stability   presented in this work is  the first result on simultaneous recovery for a class of real-valued diffusion and absorption coefficients. 

The idea of using monotonicity  and localized potentials  method  has lead to a  several  results for inverse coefficient problems; see for instance
\cite{arnold2013unique,barth2017detecting,brander2018monotonicity,griesmaier2018monotonicity,harrach2012simultaneous,harrach2018localizing,harrach2010exact,harrach2017local}.  
Together with the recent results \cite{harrach2020uniqueness,harrach2019global,harrach2019uniqueness,eberle2019lipschitz},  this work   shows that this idea can also  be used to prove Uniqueness and   Lipschitz stability results  for the inverse optical tomography problem.

Lipschitz stability   estimates for inverse and ill-posed problems   are usually based on  constructive approaches involving Carleman estimates or quantitative estimates of  unique continuation \cite{alessandrini1996determining,alessandrini2018lipschitz,
bellassoued2006lipschitz,bellassoued2007lipschitz,imanuvilov1998lipschitz,imanuvilov2001global,kazemi1993stability}.  For some applications these constructive approaches also allowed to quantify the asymptotic behavior of the Lipschitz constant; see for instance \cite{sincich2007lipschitz}.

Our approach on proving Lipschitz stability  is relatively simple compared to  previous works. The main tools are: standard (non quantitative) unique continuation, the monotonicity result and the method of localized potentials. 

For the numerical solution, we reformulate  the inverse problem into a minimization problem using a   Kohn-Vogelius functional, and  use a quasi-Newton method which employs the analytic gradient of the cost function
 and the approximation of the inverse Hessian is updated by BFGS scheme \cite{kelley1999iterative}.
Let us stress that this numerical part  approaches the problem from a heuristic numerical side to demonstrate that useful numerical reconstructions are indeed possible. It remains a challenging open task how to unite the theoretical and numerical approaches in order to find rigorously justified reconstruction methods that work well in practically relevant settings.

Let us recall  that in   \cite{klose2002optical,klose1999iterative,klose2002optical2, klose2003quasi},  the authors propose 
  new algorithms for recovering optical  material properties.  These algorithmes are 
experimentally tested for two and three- dimensional cases. While these works, which address real-life three-dimensional problems are an important step towards practical applications, they still suffer from considerable cross-talk between absorption and scattering reconstructions. What we
mean by cross-talk is that purely scattering (or purely absorbing) inclusions are often reconstructed with unphysical absorption (or scattering) properties. This behavior is well-understood from the theoretical viewpoint: Different optical distributions inside
the medium can lead to the same measurements collected at the surface of the medium
\cite{arridge1998nonuniqueness,isakov2006inverse}. To avoid such cross-talks for our numerical  results, we   have used a   suitable  regularization techniques  for the proposed algorithm  in order to better separate and estimate simultaneously the optical properties $\sigma$   and $q$. 

The paper is organized as follows. In section 2, we introduce the forward, the Neumann-to-Dirichlet operator 
and the inverse problem. Section 3 and 4 contain the main theoretical tools for this work. Section 3 is devoted to the 
reconstruction of the absorption coefficient   assuming  that   the diffusion coefficient is known.
We show a monotonicity relation and  we prove a  Runge approximation result.  Then  we deduce the existence  of localized potentials  and  prove the global uniqueness and Lipschitz stability estimate and 
 show  how to calculate the Lipschitz stability constant for a given setting.
Section 4 is concerned with the  reconstruction of the diffusion  and the absorption coefficients simultaneously.   We first  show a monotonicity result between the diffusion and   absorption coefficients  and the Neumann-to-Dirichlet operator and  prove  the existence of localized potentials. Then, we prove the Lipschitz stability estimate. In section 5, we introduce the minimization problem, and we compute the first order optimality condition.  In  section 6, satisfactory numerical results for two-dimensional problem are presented.  The last section contains  some concluding remarks.
\section{Problem  formulation}
Let $\Omega \subset \R^d$ ($d\geq 2$),  be a bounded  domain with  smooth boundary $\partial \Omega$.
For $\sigma,  q\in L^\infty_+(\Omega)$, where $L^\infty_+$ denotes the subset of $L^\infty$-functions with positive essential infima, we consider the following  problem  with Neumann boundary data  $g\in L^{2}(\partial\Omega)$:
\begin{equation}
\label{transm}
\left\{
\eqalign{
 -\nabla\cdot(\sigma \nabla u)+qu = 0\quad \textrm{ in }  \Omega,\\
\sigma\partial_{\nu} u= g\quad \textrm{ on  }\partial \Omega,
}\right.
\end{equation}
where  $\nu$ is the  unit normal vector to  $\partial\Omega$. The  weak formulation of  problem (\ref{transm})  reads
\begin{equation}
\label{eqv}
\int_\Omega \sigma \nabla u\cdot\nabla w\,dx+\int_{\Omega} q uw\,dx=\int_{\partial\Omega}gw\,ds \textrm{ for all } w\in  H^1(\Omega).
\end{equation}
Using the Riesz representation theorem (or the Lax-Milgram-Theorem), it is easily seen that 
(\ref{eqv}) is uniquely solvable and that the solution depends continuously on $g\in L^2(\partial \Omega)$ and
 $\sigma, q\in L^\infty_+(\Omega)$.
Then, we can define the   Neumann-to-Dirichlet  operator (NtD):
\[
 \eqalign{
 \Lambda(\sigma,q) : L^{2}(\partial \Omega) &\longrightarrow   L^{2}(\partial \Omega) \cr
                                & g \longmapsto u_{|\partial \Omega}, 
                               }
 \]
The  inverse problem we  consider here,  is the following:
\begin{equation}
\label{invp}
\textrm{ \it  Find the parameters  } \sigma, q \textrm{  \it  from  the knowledge of the  map }  \Lambda(\sigma,q).  
 \end{equation}
 We will consider  diffusion and absorption  parameters   that are a priori known to belong to a finite dimensional set of piecewise-analytic functions   and that are
  bounded from above and below by a priori known constants. To that end, we first define piecewise-analyticity as in \cite[Definition 2.1]{harrach2019uniqueness} 
 \begin{definition}
 \begin{itemize}
 \item[(a)] A Subset $\Gamma\subseteq \partial\Omega$ of  the boundary of an open set $\Omega\subset \R^n$ is 
 called a smooth boundary piece if it is a $C^\infty$-surface and $\Omega$ lies on one side of it, i.e. if for each 
 $z\in \Gamma$ there exists a ball $B_\epsilon(z)$ and function $\gamma\in C^\infty(\R^{n-1}, \R)$ such that 
 \[
 \Gamma=\partial\Omega\cap B_\epsilon(z)=\left\{ x\in B_\epsilon(z):  x_n=\gamma(x_1,\ldots,x_{n-1}) \right\},
 \]
 \[
 \Omega\cap B_\epsilon(z)=\left\{ x\in B_\epsilon(z):  x_n>\gamma(x_1,\ldots,x_{n-1}) \right\}.
 \]
 \item[(b)] $\Omega$  is said to have smooth boundary if $\partial\Omega$ is a union of smooth boundary pieces.
 $\Omega$ is said to have piecewise smooth boundary  if $\partial\Omega$ is a countable union of the closures of smooth  boundary pieces. 
 \item[(c)] A function $\varphi\in L^\infty(\Omega)$ is called piecewise constant if there exists finitely many
pairwise disjoint subdomains  $\Omega_1, \ldots, \Omega_N\subset \Omega$    with piecewise smooth boundaries,
such that  $\overline{\Omega}=\overline{\Omega_1\cup, \ldots,\cup\Omega_N}$ and $\varphi|_{\Omega_i}$ is constant,   $i=1,\ldots,N$.
\item[(d)] A function $\varphi\in L^\infty(\Omega)$  is called {\it piecewise analytic} if there exists finitely many pairwise 
 disjoint subdomains $\Omega_1, \ldots, \Omega_N\subset \Omega$ with piecewise smooth boundaries, such  that 
 $\overline{\Omega}=\overline{\Omega_1\cup, \ldots,\cup\Omega_N}$, and $\varphi|_{\Omega_i}$
 has an extension which is (real-)analytic in a neighborhood of $\overline{\Omega_i}$, $i=1,\ldots,N$.
 \end{itemize}
 \end{definition}
 As mentioned  in \cite{harrach2019uniqueness}, it is not clear whether the sum of two 
 piecewise-analytic functions is always piecewise-analytic, i.e. whether the set of piecewise-analytic
functions is a vector space. However, this can be guaranteed with a slightly stronger definition
of piecewise analyticity (see \cite[lemma 1]{kohn1985determining}). Therefore, we make
the following definition.
 \begin{definition}
 A set $\mathcal{F}\subseteq L^\infty(\Omega)$  is called a  finite-dimensional subset of \\
  piecewise-analytic functions if its linear span
 \[
 \textrm{span}\; \mathcal{F}=\left\{ \sum_{j=1}^k \lambda_j f_j: k\in\mathbb{N}, \lambda_j\in\R,  f_j \in\mathcal{F}  \right\}
 \subseteq L^\infty(\Omega,
 \]
  contains only piecewise-analytic functions and dim(span $\mathcal{F})<\infty$.
  \end{definition}
 \section{Recovery of the absorption coefficient}
 In this section, we assume that   $\sigma=\sigma_0\chi_{\Omega\setminus\omega}+\sigma_1\chi_{\omega}$,  
 and $q=q\chi_{\omega}$,  where    $\sigma_0,\sigma_1$ are positive constants and $\omega\Subset \Omega$.
We aim to recover the  absorption parameter $q\in L^\infty_+(\omega)$  from  the NtD  operator  
 \[
 \Lambda(q):  L^2(\partial\Omega)\to L^2(\partial\Omega):  g\mapsto  u\vert_{\partial\Omega}.
 \]
 provided that $\sigma$ is  known.
 
 Given a finite-dimensional subset $\mathcal{F}$ of piecewise analytic functions and  two  constants 
$b > a > 0$,  we denote the set
 \[
\mathcal{F}_{[a,b]}:=\left\{ q\in \mathcal{F}:\quad   a\leq q(x)\leq b, \quad \textrm{ for all  } x\in \omega \right\}.   
\]
Throughout this paper, the domain  $\omega$,  the finite-dimensional subset $\mathcal{F}$ and the bounds
$b > a > 0$ are fixed, and the constants in the Lipschitz stability results will depend on them.
Our first results show Uniqueness and  Lipschitz stability for the inverse  absorption  problem in $\mathcal{F}_{[a,b]}$, when the complete infinite-dimensional NtD-operator is measured.\\
\noindent
The outline of this section  is the following
\begin{itemize}
\item[$(i)$] In Subsection  \ref{sec_runge}, we prove a runge approximation result   and we deduce a global  uniqueness  for determining $q$ from $\Lambda(q)$.
\item[$(ii)$] In Subsection  \ref{mon_local}, we   show  a monotonicity  and localized potentials results and we deduce a Lipschitz stability  estimate for determining $q$ from $\Lambda(q)$.
\item[$(iii)$] In Subsection \ref{quant}, we show how to quantify  the Lipschitz constant.
\end{itemize}
\subsection{Runge approximation and uniqueness.}\label{sec_runge}
We first note the following unique continuation property.     For every
open connected subset $\mathcal{O} \subset  \Omega $,  only the trivial solution of
\[
-{\rm div}(\sigma\nabla u)+ qu=0\textrm{  in } \mathcal{O},
\]
vanishes on an open subset of $\mathcal{O}$ or possesses zero Cauchy data on a smooth, open
part of $\partial \mathcal{O}$. When $\sigma$  is Lipschitz and $q$  is bounded,   this property is proven in
Miranda \cite[Thm. 19, II]{miranda2013partial}. It  can be extended to the  case of piecewise analytic $\sigma$ and  $q$ by sequentially solving Cauchy problems (see \cite{druskin1998uniqueness}). 

 We will deduce the uniqueness theorem \ref{uniqueness} from the following Runge approximation result.
\begin{theorem}[Runge approximation]\label{thm:runge}
Let $q\in L^\infty_+(\omega)$ be piecewise analytic. For all $f\in L^2(\omega)$ there exists a sequence $(g_n)_{n\in\N}\subset L^2(\partial\Omega)$ such that the corresponding solutions $u^{(g_n)}$ of (\ref{transm}) with boundary data $g_n$, $n\in \N$, fulfill
\[
u^{(g_n)}|_\omega\to f \quad \textrm{ in } L^2(\omega).
\]
\end{theorem}
\begin{proof}
We introduce the operator
\[
A:  L^2(\omega)\to  L^2(\partial\Omega), \quad f\mapsto A f:= v_{|\partial\Omega}, 
\]
where $v\in H^1(\Omega)$  solves 
\begin{equation}\label{eq:runge_v}
\int_\Omega\sigma\nabla v\cdot\nabla w\,dx+\int_\omega q v w\,dx=\int_{\omega}f w\,dx\quad \textrm{ for all } w\in H^1(\Omega).
\end{equation}
Let  $g\in L^2(\partial \Omega)$ and $u\in H^1(\Omega)$ be the corresponding solution  of problem (\ref{transm}).  Then the adjoint operator of  $A$ is characterized by 
\begin{equation}
\eqalign{
 \int_{\omega} \left( A^* g \right) f \,dx = \int_{\partial \Omega} \left( A f \right) g \,ds = \int_{\partial \Omega} v g \,ds 
=\int_\Omega \sigma\nabla u\cdot\nabla v\,dx+\int_\omega q uv\,dx\cr
=\int_{\omega}f u\,dx, \quad \textrm{ for all } f\in L^2(\omega),
}
\end{equation}
which shows that $A^*:\ L^2(\partial \Omega)\to L^2(\omega)$ fulfills 
$A^* g=u|_{\omega}$. The assertion follows if we can show that $A^*$ has dense range, which is equivalent
to $A$ being injective.

To prove this, let $v|_{\partial \Omega}=Af=0$ with $v\in H^1(\Omega)$ solving (\ref{eq:runge_v}).
Since (\ref{eq:runge_v}) also implies that $\sigma\partial_\nu v|_{\partial \Omega}=0$, and $\Omega\setminus\omega$ is connected, it follows by unique continuation that $v|_{\Omega\setminus\omega}=0$ and thus $v^+|_{\partial\omega}=0$. Since $v\in H^1(\Omega)$ this also implies that
$v^-|_{\partial\omega}=0$, and together with (\ref{eq:runge_v}) we obtain that $v|_{\omega}\in H^1(\omega) $ solves
\[
-\nabla\cdot (\sigma \nabla v)+qv=0 \quad \textrm{ in } \omega,
\]
with homogeneous Dirichlet boundary data $v|_{\partial \omega}=0$. Hence, $v|_{\omega}=0$, so that $v=0$ almost everywhere in $\Omega$.
From (\ref{eq:runge_v}) it then follows that $\int_{\omega}f w\,dx=0$ for all $w\in H^1(\Omega)$ and thus $f=0$.
\end{proof}
\begin{theorem}[Global uniqueness]
\label{uniqueness}
For $q_1, q_2\in L^\infty_+(\omega)$  that are piecewise analytic, 
\[
\Lambda(q_1)=\Lambda(q_2)\quad \textrm{if and only if}\quad q_1=q_2.
\] 
\end{theorem}
\begin{proof}
For absorption  parameters $q_1, q_2\in L^\infty_+(\omega)$ and Neumann data $g,h\in L^2(\partial \Omega)$ we denote the
corresponding solutions of (\ref{transm}) by $u_{1}^g$, $u_1^h$, $u_2^g$, and $u_2^h$ respectively. 
The variational formulation (\ref{eqv}) yields the orthogonality relation
\[
\eqalign{
 &\lefteqn{\int_{\partial \Omega} h \left( \Lambda(q_2)-\Lambda(q_1)\right) g \, ds}\cr
&= \int_{\partial \Omega} h \Lambda(q_2) g \, ds  - \int_{\partial \Omega} g \Lambda(q_1) h \, ds
= \int_{\partial \Omega} h u_2^g \, ds  - \int_{\partial \Omega} g u_1^h \, ds\cr
&= \int_\Omega\sigma\nabla u_1^h \cdot\nabla u_2^g \,dx+\int_\omega q_1 u_1^h u_2^g\, dx
- \left( \int_\Omega\sigma\nabla u_2^g \cdot\nabla u_1^h \,dx+\int_\omega q_2 u_2^g u_1^h\, dx\right)\cr
&= \int_\omega (q_1-q_2) u_1^h u_2^g\, dx. 
}
\]
This shows that $\Lambda(q_1)=\Lambda(q_2)$ implies that
\[
\int_\omega (q_1-q_2) u_1^h u_2^g\, dx=0,  \quad \textrm{ for all } g,h\in L^2(\partial \Omega).
\]
Using the Runge approximation result in theorem~\ref{thm:runge}, this yields that 
$(q_1-q_2) u_1^h=0$ (a.e.) in $\omega$ for all $h\in L^2(\partial \Omega)$, and using theorem~\ref{thm:runge} again, this implies
$q_1=q_2$. 
\end{proof}
\subsection{Monotonicity, localized potentials and Lipschitz stability}\label{mon_local}
To prove the Lipschitz stability result in Theorem \ref{stability2}, we first show a monotonicity estimate between the absorption coefficient and the Neumann-to-Dirichlet operator, and deduce the existence of localized potentials from the Runge approximation result.
\begin{lemma}[Monotonicity estimate]
\label{mono}
Let $q_1, q_2\in L^\infty_+(\omega)$ be two  absorption parameters, let $g\in L^2(\partial\Omega)$ be an applied boundary current, and let $u_2:=u^{g}_{q_2}\in H^1(\Omega)$ solve (\ref{transm}) for the boundary current $g$ and the absorption parameter $q_2$. Then
\begin{equation}
\label{eqmono}
 \int_\omega(q_1-q_2) u_2^2\,dx\geq \int_{\partial \Omega} g \left(\Lambda(q_2)-\Lambda(q_1)\right) g\, ds\geq 
\int_\omega\left(q_2-\frac{q^2_2}{q_1}  \right)u^2_2\,dx.
\end{equation}
\end{lemma}
\begin{proof}
Let $u_1:=u^g_{q_1}\in H^1(\Omega)$. From the variational equation,  we deduce 
\[
 \int_\Omega\sigma \nabla u_1\cdot\nabla u_2\,dx+\int_\omega q_1 u_1u_2\,dx=
\int_{\partial \Omega} g \Lambda(q_2)g\, ds=\int_\Omega\sigma |\nabla u_2|^2\,dx+\int_\omega q_2 u_2^2\,dx.
\]
Thus 
\[
\eqalign{
 \int_\Omega\sigma |\nabla(u_1-u_2)|^2\,dx&+\int_\omega q_1(u_1-u_2)^2\,dx \cr
&=\int_\Omega\sigma |\nabla u_1|^2\,dx+ \int_\omega q_1u_1^2\,dx
+\int_\Omega\sigma |\nabla u_2|^2\,dx+ \int_\omega q_1u_2^2\,dx\cr
&-2\int_\Omega\sigma |\nabla u_2|^2\,dx-2\int_\omega q_2u^2_2\,dx\cr
&=\int_{\partial \Omega} g \Lambda(q_1)g\, ds
-\int_{\partial \Omega} g \Lambda(q_2)g\, ds
+\int_\omega (q_1-q_2)u^2_2\,dx.
}
\]
Since the left-hand side is nonnegative, the first asserted inequality follows. 
\noindent
Interchanging $q_1$ and $ q_2$, we  get 
\[
\eqalign{
 &\int_{\partial \Omega} g \Lambda(q_2)g\, ds-\int_{\partial \Omega} g \Lambda(q_1)g\,ds\cr
&=\int_\Omega\sigma |\nabla(u_2-u_1)|^2\,dx+\int_\omega q_2(u_2-u_1)^2\,dx
-\int_\omega (q_2-q_1)u^2_1\,dx\cr
&=\int_\Omega\sigma |\nabla(u_2-u_1)|^2\,dx+\int_\omega\left(q_2u^2_2-2q_2u_1u_2+q_1u^2_1\right)\,dx\cr
&=\int_\Omega\sigma |\nabla(u_2-u_1)|^2\,dx+\int_\omega q_1\left(u_1-\frac{q_2}{q_1}u_2 \right)^2\,ds+\int_\omega\left(q_2-\frac{q^2_2}{q_1} \right)u^2_2\,dx.
}
\]
Since the first two integrals on the right-hand side are non negative, the second asserted inequality follows. 
\end{proof}
\noindent
Note that we call Lemma \ref{mono} a monotonicity estimate  because of the following  corollary:
\begin{corollary}[Monotonicity]
For two absorption parameters $q_1, q_2 \in L^\infty_+(\omega)$
\begin{equation}
 q_1\leq q_2 \quad \textrm{implies }\quad \Lambda(q_1)\geq \Lambda(q_2)\quad \textrm{in the sense of quadratic forms}.
\end{equation}
\end{corollary}
\noindent
Let us stress, however, that Lemma \ref{mono} holds for any $q_1, q_2\in L^\infty_+(\omega)$ and does not require $q_1\leq q_2$ or $q_1\geq q_2$. \\
\noindent
The existence of localized potentials  follows from the Runge approximation property as in \cite[Lemma~4.3]{harrach2019global}.
\begin{lemma}[Localized potentials]\label{lemma:locpot1}
Let $ q \in L^\infty_+(\omega)$ be  piecewise analytic, and let  $\mathcal{O}\subseteq \omega$ be a subset with positive boundary measure. 
Then there exists a sequence $(g_n)_{n\in\N}\subset L^2(\partial\Omega)$ such that the corresponding solutions $u^{(g_n)}$ of (\ref{transm}) fulfill
\[
\lim_{n\to \infty}\int_{\mathcal{O}} |u^{(g_n)}|^2\,ds=\infty \quad \textrm{ and }\quad
\lim_{n\to \infty}\int_{\omega\setminus \mathcal{O}} |u^{(g_n)}|^2\,ds=0.
\]
\end{lemma}
\begin{proof}
Using the Runge approximation property in Theorem~\ref{thm:runge}, we find a sequence $\tilde g_n\in L^2(\partial \Omega)$
so that the corresponding solutions $u^{(\tilde g_n)}$ fulfill
\[
u^{(\tilde g_n)}|_\omega \to \frac{\chi_{\mathcal{O}}}{\left(\int_{\mathcal{O}}\,dx\right)^{1/2}}\quad \textrm{ in } L^2(\omega).
\]
Hence
\begin{equation*}
\lim_{n\to \infty}\int_{\mathcal{O}} |u^{(\tilde g_n)}|^2\,dx=1  \quad \textrm{ and }\quad
\lim_{n\to \infty}\int_{\omega\setminus \mathcal{O}} |u^{(\tilde g_n)}|^2\,dx=0,
\end{equation*}
so that 
\[
g_n:=\frac{\tilde g_n}{\left(\int_{\omega\setminus \mathcal{O}}\tilde u_n^2\,dx\right)^{1/4}}, 
\]
has the desired property
\[
\eqalign{
\lim_{n\to \infty}\int_{\mathcal O} |u^{(g_n)}|^2\,dx &= \lim_{n\to \infty} \frac{\int_{\mathcal O} |u^{(\tilde g_n)}|^2\,dx}{\left(\int_{\omega\setminus \mathcal O}
|u^{(\tilde g_n)}|^2\,dx\right)^{1/2}}=\infty,\cr
\lim_{n\to \infty}\int_{\omega\setminus \mathcal O} |u^{(g_n)}|^2\,dx &= \lim_{n\to \infty} \left(\int_{\omega\setminus \mathcal O}|u^{(\tilde g_n)}|^2\,dx\right)^{1/2} = 0.
}
\]
\end{proof}
 \begin{theorem}[Lipschitz stability]
\label{stability}
There exists a constant $C>0$ such that    
\[
\| q_1-q_2 \|_{L^{\infty}(\omega)}\leq   C\| \Lambda(q_1)-\Lambda(q_2) \|_{\mathcal L(L^2(\partial \Omega))},
\quad  \textrm{ for all } q_1, q_2\in \mathcal{F}_{[a,b]}.
\]
\end{theorem}
\begin{proof}
Let $\mathcal{F}\subset L^\infty(\omega)$ be a finite dimensional subspace  of piecewise analytic functions, $b>a>0$, and 
\[
q_1,  q_2\in \mathcal{F}_{[a,b]}=\left\{q\in \mathcal{F}:\quad   a\leq q(x)\leq b \textrm{ for  all } x\in \omega\right\}.   
\]
For the ease of notation, we write in the following
\[
\Vert q_1-q_2\Vert:=\Vert q_1-q_2\Vert_{L^\infty(\Omega)} \quad \textrm{ and } \quad
\Vert g\Vert :=\Vert g \Vert_{L^2(\partial \Omega)}.
\]

Since $\Lambda(q_1)$ and $\Lambda(q_2)$ are self-adjoint, we have that
\[
\eqalign{
\lefteqn{\norm{\Lambda(q_2)-\Lambda(q_1)}_*}\cr
&= \sup_{\norm{g}=1} \left| \int_{\partial \Omega} g \left(\Lambda(q_2)-\Lambda(q_1)\right) g\, ds \right|\cr
  &=  \sup_{\norm{g}=1} \max\left\{
\int_{\partial \Omega} g \left(\Lambda(q_2)-\Lambda(q_1)\right) g\, ds,
\int_{\partial \Omega} g \left(\Lambda(q_1)-\Lambda(q_2)\right) g\, ds
\right\}.
}
\]
Using the first inequality in the monotonicity relation (\ref{eqmono}) in Lemma \ref{mono} in its original form, and
with $q_1$ and $q_2$ interchanged, we obtain for all $g\in L^2(\partial \Omega)$
\[
\eqalign{
\int_{\partial \Omega} g \left(\Lambda(q_2)-\Lambda(q_1)\right) g\, ds &\geq 
\int_{\omega} (q_1-q_2)|u_{q_1}^{(g)}|^2\, dx,\cr
\int_{\partial \Omega} g \left(\Lambda(q_1)-\Lambda(q_2)\right) g\, ds &\geq 
\int_{\omega}(q_2-q_1)|u_{q_2}^{(g)}|^2\,dx,
}
\]
where $u_{q_1}^{(g)},u_{q_2}^{(g)}\in H^1(\Omega)$ denote the solutions of (\ref{transm}) with Neumann data $g$ and absorption parameter $q_1$ and $q_2$, resp.
Hence, for $q_1\neq q_2$, we have
\[
\frac{\| \Lambda(q_2)-\Lambda(q_1) \|_*}{\| q_1 -q_2\|}\geq  
 \sup_{\| g \|=1}\phi\left(g,\frac{q_1-q_2}{\| q_1 -q_2\|_{L^\infty(\omega)}},q_1, q_2\right),
\]
where (for $g\in L^2(\partial \Omega)$, $\zeta\in \mathcal{F}$, and $\kappa_1,\kappa_2\in \mathcal{F}_{[a,b]}$)
\begin{equation}\label{eq:Definition_Psi}
\phi\left(g,\zeta,\kappa_1,\kappa_2\right):=
\max \left\{ \int_{\omega} \zeta |u_{\kappa_1}^{(g)}|^2 \,dx, \int_{\omega} (-\zeta) |u_{\kappa_2}^{(g)}|^2 \,dx  \right\}.
\end{equation}
Introduce the compact set
\begin{equation}\label{eq:Definition_CC}
\mathcal{C}=\left\{ \zeta\in \textrm{ span} \;\mathcal{F}:\quad \| \zeta \|_{L^\infty(\omega)}=1 \right\}.
\end{equation}
Then, we have
\begin{equation}\label{eq:stability_infsup1}
\eqalign{
\frac{\Vert\Lambda(q_2)-\Lambda(q_1) \Vert_*}{\Vert q_1 -q_2\Vert}&\geq \sup_{\Vert g \Vert=1} \phi(g,\zeta,\kappa_1,\kappa_2)\cr
&\geq \inf_{\substack{\zeta\in \mathcal{C}\cr \kappa_1,\kappa_2\in\mathcal{F}_{[a,b]}}} \sup_{\Vert g \Vert=1} \phi(g,\zeta,\kappa_1,\kappa_2).
}
\end{equation}
The assertion of Theorem \ref{stability} follows if we can show that the right hand side of (\ref{eq:stability_infsup1})
is positive. Since $\phi$ is continuous, the function
\[
(\zeta,\kappa_1,\kappa_2)\mapsto \sup_{\| g \|=1} \phi(g,\zeta,\kappa_1,\kappa_2)
\]
is semi-lower continuous, so that it attains its minimum on  the compact set 
$\mathcal{C}\times\mathcal{F}_{[a,b]}\times \mathcal{F}_{[a,b]}$.
Hence, to prove Theorem \ref{stability}, it suffices to show that
\[
\sup_{\| g \|=1} \phi(g,\zeta,\kappa_1,\kappa_2)>0 \quad \text{ for all } (\zeta,\kappa_1,\kappa_2)\in \mathcal{C}\times\mathcal{F}_{[a,b]}\times \mathcal{F}_{[a,b]}.
\]
To show this, let $(\zeta,\kappa_1,\kappa_2)\in \mathcal{C}\times\mathcal{F}_{[a,b]}\times \mathcal{F}_{[a,b]}$.
Since $\norm{\zeta}_{L^\infty(\omega)}=1$, there exists a subset $\mathcal O\subseteq \omega$ with positive  measure  and  $0<\Theta<1$   such that
either
\[
\text{(a)}\ \zeta(x)\geq \Theta\text{ for all } x\in \mathcal O,\quad \text{ or } \quad
\text{(b)}\ -\zeta(x)\geq \Theta\text{ for all } x\in \mathcal O.
\]
In case (a), we use the localized potentials sequence in Lemma~\ref{lemma:locpot1}, to obtain a boundary current
$\hat g\in L^2(\partial \Omega)$ with 
\[
\int_{\mathcal O} \left|u^{(\hat g)}_{\kappa_1}\right|^2\, dx \geq \frac{1}{\Theta} \quad \text{ and } \quad \int_{\omega\setminus \mathcal O} \left|u^{(\hat g)}_{\kappa_1}\right|^2\, dx \leq \frac{1}{2},
\]
so that (using again $\norm{\zeta}_{L^\infty(\omega)}=1$)
\[
\phi\left(\hat g,\zeta,\kappa_1,\kappa_2\right)\geq \int_\omega \zeta \left|u^{(\hat g)}_{\kappa_1}\right|^2\, dx \geq \Theta \int_{\mathcal O} \left|u^{(\hat g)}_{\kappa_1}\right|^2\, dx
- \int_{\omega\setminus \mathcal O} \left|u^{(\hat g)}_{\kappa_1}\right|^2\, dx \geq \frac{1}{2}.
\]
In case (b), we can analogously use a localized potentials sequence for $\kappa_2$, and 
find $\hat g\in L^2(\partial \Omega)$ with
\[
\phi\left(\hat g,\zeta,\kappa_1,\kappa_2\right)\geq \int_\omega (-\zeta) \left|u^{(\hat g)}_{\kappa_2}\right|^2\, dx \geq \Theta \int_{\mathcal O} \left|u^{(\hat g)}_{\kappa_2}\right|^2\, dx
- \int_{\omega\setminus \mathcal O} \left|u^{(\hat g)}_{\kappa_2}\right|^2\, dx \geq \frac{1}{2}.
\]
Hence, in both cases, 
\[
\sup_{\| g \|=1} \phi(g,\zeta,\kappa_1,\kappa_2)\geq  \phi\left(\frac{\hat g}{\norm{\hat g}} ,\zeta,\kappa_1,\kappa_2\right)=
\frac{1}{\norm{\hat g}^2} \phi(\hat g,\zeta,\kappa_1,\kappa_2)>0,
\]
so that Theorem \ref{stability} is proven.
\end{proof}
\subsection{Quantitative Lipschitz stability} \label{quant}
In this subsection, we  restrict ourself   to  the case where  $\mathcal{F}$  is a  set of piecewise constant  functions  
  on   a given  partition $\cup_{j=1}^N D_j=\omega$, i.e,
\[
\mathcal{F}=\left\{q(x)=\sum_{j=1}^N q_j\chi_{D_j},\quad  q_1,\ldots q_N\in\R\right\}\subset L^\infty(\omega),
\]
  and for  $0<a< b$,   $\mathcal{F}_{[a,b]}$
is the set  of $q\in \mathcal{F}$  such  that $a\leq q_j\leq  b$  for all $j=1,\ldots, N$.
The structure   assumed   for $q$  fits well in several  problems  arising in practical applications.\\
\noindent
For our   quantitative Lipschitz stability estimate,  we need a finite numbers of localized potentials  and we show how 
to  reconstruct them.
\begin{lemma}\label{existence}
Let $b>a>0$ be given constants. For $j=1,\ldots, N$ anf $k=1,\ldots,K$,  with $K=\left(\lfloor3\left( \frac{b}{a}-1 \right)\rfloor +3 \right)$, we  define the piecewise constant function $\eta^{(j,k)}\in L_{+}^\infty(\omega)$  by
\[
\eta^{(j,k)}(x)=\left\{
\eqalign{
(k+4)\frac{a}{3}\quad &\text{ if }  x\in   D_j,\\
\frac{a}{3}\quad &\text{ if }  x\in   \omega\setminus D_j.
}
\right.
\]
\begin{itemize}
\item[(i)] There exist  boundary data  $g^{(j,k)}\in \L^2(\partial\Omega)$,  so  that the corresponding  solutions
$u^{g^{(j,k)}}_{\eta^{(j,k)}}\in H^1(\Omega)$ of  (\ref{transm})  with  $g= g^{(j,k)}$  and $q=\eta^{(j,k)}$  fulfill
\begin{equation}
\label{exist_g}
 \beta^{(i,k)}:= \frac{1}{2}\int_{D_j} \vert u^{g^{(j,k)}}_{\eta^{(j,k)}}\vert^2\,dx-\left( \frac{3b}{2a} -\frac{1}{2} \right)\int_{\omega\setminus D_j}
\vert u^{g^{(j,k)}}_{\eta^{(j,k)}}\vert^2\,dx>1.
\end{equation}
\item[(ii)]
For arbitrary  $q\in \mathcal{F}_{[a,b]}$,  the solutions    $u^{g^{(j,k)}}_q\in H^1(\Omega)$ of  (\ref{transm})  with  $g= g^{(j,k)}$ fulfill
\[
 \int_{D_j}\vert u^{g^{(j,k)}}_q \vert^2\,dx  - \int_{\omega\setminus D_j}\vert u^{g^{(j,k)}}_q \vert^2\,dx \geq  \beta^{(j,k)} > 
1.
\]
\item[(iii)]   $g^{(j,k)}$ can be computed  by solving a finite number of well-posed PDEs.
\end{itemize}
\end{lemma}
\begin{proof}
 $(i)$ follows immediately  from the localized potentials  result in   Lemma \ref{lemma:locpot1}. To prove $(b)$,  we  need the following monotonicity  result which follows  from  Lemma \ref{mono} with $q_2=q+\delta$  and $q_1=q$
 and from using the same inequality again with interchanged roles of $q_1$  and $q_2$.
 For $g\in L^2(\partial\Omega)$, $q\in L^{\infty}_+(\omega)$,    and $\delta \in L^{\infty}_+(\omega) $   such  that 
 $(q+\delta) \in L^{\infty}_+(\omega)$,  we have
 \begin{equation}\label{monob}
 \int_\omega \delta u^g_q\,dx \geq    \int_\omega \delta u^g_{q+\delta}\,dx.
 \end{equation}
 Let $j=1,\ldots,N$ and   $q\in \mathcal{F}_{[a,b]}$.  Since $K$  fullfils  $b< (K+3)\frac{a}{3}$,  there exists $k\in \{1,\ldots, K \}$
 such   that  $q_j=  q|_{D_j}$  fullfils
 \[
(k+2 ) \frac{a}{3}\leq  q_j<  (k+3) \frac{a}{3}.
 \]
 Using the monotonicity-based inequality (\ref{monob}),   with  
 \[
  \frac{a}{3}\leq  (k+4 ) \frac{a}{3}- q_j<  \frac{2a}{3} \quad  \text{ and } \quad -b+ \frac{a}{3}\leq   \frac{a}{3}- q_j<  -\frac{2a}{3},
  \]
  we obtain 
  \[
  \eqalign{
  &  \int_{D_j}\vert u^{g^{(j,k)}}_q \vert^2\,dx  - \int_{\omega\setminus D_j}\vert u^{g^{(j,k)}}_q \vert^2\,dx \\
& = \frac{3}{2a}\left(\int_{D_j} \frac{2a}{3}\vert u^{g^{(j,k)}}_q \vert^2\,dx 
   -  \frac{2a}{3}\int_{\omega\setminus D_j}\vert u^{g^{(j,k)}}_q \vert^2\,dx\right) \\
  & \geq \frac{3}{2a}\left(\int_{D_j} \left( (k+4 ) \frac{a}{3}- q_j\right) \vert u^{g^{(j,k)}}_q \vert^2\,dx  +\int_{\omega\setminus D_j}\left(\frac{a}{3}-q_j\right)\vert u^{g^{(j,k)}}_q \vert^2\,dx\right)\\
  &= \frac{3}{2a}\left(\int_{D_j} \left(\eta^{(j,k)}- q_j\right) \vert u^{g^{(j,k)}}_q \vert^2\,dx  +\int_{\omega\setminus D_j}\left(\eta^{(j,k)}-q_j\right)\vert u^{g^{(j,k)}}_q \vert^2\,dx\right)\\
  & \geq \frac{3}{2a}\left(\int_{D_j} \left(\eta^{(j,k)}- q_j\right) \vert u^{g^{(j,k)}}_{\eta^{(j,k)}} \vert^2\,dx  +\int_{\omega\setminus D_j}\left(\eta^{(j,k)}-q_j\right)\vert u^{g^{(j,k)}}_{\eta^{(j,k)}} \vert^2\,dx\right)\\
  &\geq \frac{3}{2a}\left(\int_{D_j} \frac{a}{3} \vert u^{g^{(j,k)}}_{\eta^{(j,k)}} \vert^2\,dx  -\int_{\omega\setminus D_j}\left(b-\frac{a}{3}\right)\vert u^{g^{(j,k)}}_{\eta^{(j,k)}} \vert^2\,dx\right)\\
  &= \frac{1}{2} \int_{D_j}  \vert u^{g^{(j,k)}}_{\eta^{(j,k)}} \vert^2\,dx-\left( \frac{3b}{2a} -\frac{1}{2} \right)
  \int_{\omega\setminus D_j}\vert u^{g^{(j,k)}}_{\eta^{(j,k)}} \vert^2\,dx=\beta^{(j,k)}>1,
  }
  \]
  and $(ii)$ is   proved.  To prove  $(iii)$  we use a similar approach  as in the construction of localized potentials in   \cite{harrach2019global}. For $j=1,\ldots, N$  and  $k=1,\ldots,K$,   we introduce the operator $A$  as in Theorem \ref{thm:runge} 
 \[
A:  L^2(\omega)\to  L^2(\partial\Omega), \quad f\mapsto A f:= v_{|\partial\Omega}, 
\]
where $v\in H^1(\Omega)$  solves 
\[
\int_\Omega\sigma\nabla v\cdot\nabla w\,dx+\int_\omega \eta^{(j,k)} v w\,dx=\int_{\omega}f w\,dx\quad \textrm{ for all } w\in H^1(\Omega).
\]
We have shown  that the adjoint operator  $A^*$ of   $A$  is given  by 
\[
A^*:  L^2(\partial\Omega)\rightarrow L^2(\omega):  g\mapsto  u\vert_{\omega},
\]
 where $u$  is the solution of   (\ref{transm})  with  $q=\eta^{(j,k)}$,  and that $A^*$  has dense range.\\
 \noindent
 Consider the linear  ill-posed equation
 \[
 A^*g=3\chi_{D_j}.
 \]
 Sine $ 3\chi_{D_j} \in \overline{\mathcal{R}(A^*)}$,  the conjugate gradient method  \cite[III.15]{hanke2017taste}, yields  a sequence  of iterates 
 $(g_n)_{n\in \mathbb{N}}\subset L^2(\partial\Omega)$      for  which 
 \[ 
 A^*g_n \rightarrow 3\chi_{D_j}.
 \] 
Therefore,  the solutions $u_n$  of   (\ref{transm})    with  $q=\eta^{(j,k)}$  and  $g=g_n$  fulfill
\[
 \frac{1}{2}\int_{D_j} \vert u_n\vert^2\,dx-\left( \frac{3b}{2a} -\frac{1}{2} \right)\int_{\omega\setminus D_j}
\vert u_n\vert^2\,dx\rightarrow \frac{3}{2},
\]
so that after finitely many iteration steps, (\ref{exist_g}) is fulfilled.
\end{proof}
\noindent
Now, we state the main result of this  subsection.
 \begin{theorem}[Quantitative Lipschitz stability]\label{quantitative}
 Let $g^{(j,k)}\in L^2(\Omega)$  defined as in Lemma  \ref{existence}.  Set 
 \[
 L=\left(\max\left\{  \Vert  g^{(j,k)}\Vert^2_{L^2(\partial\Omega)}, \quad j=1,\ldots, N,  k=1,\ldots, K  \right\}\right)^{-1}.
 \]
Then 
 \begin{equation}\label{stab_q}
 \Vert  q_1-q_2  \Vert_{\infty}\leq L\Vert \Lambda(q_1)-\Lambda(q_2) \Vert_{\infty} \quad \text{ for all }\quad
  q_1,  q_2\in \mathcal{F}_{[a,b]}.
 \end{equation}
 \end{theorem}
 \begin{proof}
 From  Lemma \ref{existence},  we have for all  $q\in \mathcal{F}_{[a,b]}$,  and for all $j\in \{1,\ldots,N\}$
 \begin{equation}\label{eq_existence}
 \eqalign{
 &\sup_{\Vert g\Vert=1}\left( \int_{D_j}\vert u^g_q \vert^2\,dx - \int_{\omega\setminus D_j}\vert u^g_q \vert^2\,dx\right)\\
&= \sup_{0\neq g\in L^2(\partial \Omega)}\frac{1}{\Vert g\Vert^2}\left( \int_{D_j}\vert u^g_q \vert^2\,dx- \int_{\omega\setminus D_j}\vert u^g_q \vert^2\,dx\right)\geq L. 
}
 \end{equation}
 To prove Theorem \ref{quantitative},  it suffices  to show  that  
  \begin{equation}\label{eq:phi}
  \sup_{\| g \|=1} \phi(g,\zeta,\kappa_1,\kappa_2)\geq L \quad \text{ for all } (\zeta,\kappa_1,\kappa_2)\in \mathcal{C}\times\mathcal{F}_{[a,b]}\times \mathcal{F}_{[a,b]},
  \end{equation}
   where  $\phi$  and $\mathcal{C}$ defined in  (\ref{eq:Definition_Psi}) and  (\ref{eq:Definition_CC}).  
 Since $\mathcal{F}$ contains only piecewise-constant functions,  for every $\zeta \in  \mathcal{C}$ there
must exist a  subset  $D_j\subset \omega$ with either
 \[
 \zeta\vert_{Dj}=1,  \quad \text{  or }\quad    \zeta\vert_{Dj}=-1, 
 \]
 Hence using   (\ref{eq:Definition_Psi}) and  (\ref{eq_existence}),  we obtain for the case $\zeta\vert_{Dj}=1$,
 \[
 \eqalign{
  \sup_{\| g \|=1} \phi(g,\zeta,\kappa_1,\kappa_2)&\geq   \sup_{\| g \|=1}\int_\omega\zeta \vert u^g_{\kappa_1} \vert^2\,dx\\
& \geq  \sup_{\| g \|=1}\left(  \int_{D_j} \vert u^g_{\kappa_1} \vert^2\,dx-
 \int_{\omega\setminus D_j} \vert u^g_{\kappa_1} \vert^2\,dx
 \right)\geq L,
 }
 \]
 and for the case  $\zeta\vert_{Dj}=-1$,
  \[
  \eqalign{
  \sup_{\| g \|=1} \phi(g,\zeta,\kappa_2,\kappa_2)&\geq   \sup_{\| g \|=1}\int_\omega(-\zeta) \vert u^g_{\kappa_2} \vert^2\,dx\\
& \geq  \sup_{\| g \|=1}\left(  \int_{D_j} \vert u^g_{\kappa_2} \vert^2\,dx-
 \int_{\omega\setminus D_j} \vert u^g_{\kappa_2} \vert^2\,dx
 \right)\geq L.
 }
 \]
 so that (\ref{eq:phi}) is proved and the proof is completed.
  \end{proof}
\section{Simultaneous recovery of diffusion and absorption}
The inverse problem  of recovering $\sigma$ and $q$  simultaneously  is known to be  an ill-posed problem
and stability results can only be obtained under a-priori assumptions.  

For our problem, we will prove a stability result under the assumption that
the coefficients  belong to an a-priori known finite-dimensional subspace, that upper
and lower bounds are a-priori known, and that a  definiteness condition holds. 

As in the last section the main tools to prove the  stability are the monotonicity  and  the existence of localized 
potentials, which are the subject of the following  subsection.
\subsection{Monotonicity and localized potentials}
\begin{lemma}[Monotonivity]
\label{mono2}
Let $\sigma_1, \sigma_2, q_1, q_2\in L^\infty_+(\Omega)$. Then
\begin{equation}
\label{eqmono1}
\eqalign{
 \int_\Omega\bigl[(\sigma_2-\sigma_1)|\nabla u_1|^2+(q_2-q_1) u_1^2\bigr]\,dx&\geq 
\langle g, \left(\Lambda(\sigma_1,q_1)-\Lambda(\sigma_2,q_2)\right)g\rangle\cr
&\geq \int_\Omega\bigl[(\sigma_2-\sigma_1)|\nabla u_2|^2+(q_2-q_1) u_2^2\bigr]\,dx,
}
\end{equation}
\begin{equation}
\label{eqmono2}
\eqalign{
 \langle g, \left(\Lambda(\sigma_1,q_1)-\Lambda(\sigma_2,q_2)\right) g\rangle &\geq 
\int_\Omega\left[\left(\sigma_1-\frac{\sigma_1^2}{\sigma_2}\right)\vert\nabla u_1\vert^2+\left(q_1-\frac{q_1^2}{q_2}\right) u_1^2\right]\,dx\cr
&=\int_\Omega\left[\frac{\sigma_1}{\sigma_2}(\sigma_2-\sigma_1)|\nabla u_1|^2+\frac{q_1}{q_2}(q_2-q_1) u_1^2\right]\,dx,
}
\end{equation}
for all  $g \in  L^2(\partial\Omega)$ where $u_1,  u_2 \in H^1(\Omega)$ are the solutions of (\ref{transm}) with Neumann boundary data $g$ on $\partial\Omega$, 
and coefficients $(\sigma_1,  q_1)$, resp., $(\sigma_2, q_2)$.
\end{lemma}
\begin{proof}
The proof of (\ref{eqmono1})    is given in \cite[Lemma\; 4.1 ]{harrach2009uniqueness}. Following the proof 
of Lemma \ref{mono}, we can easily   deduce (\ref{eqmono2}). 
\end{proof}
\begin{theorem}[Localized potentials]
\label{thm:locpot}
Let $\sigma, q \in L^\infty_+(\Omega)$ that are  piecewise analytic  and  $D\Subset\Omega $ be  non empty  open set,   such that   $\Omega\setminus\overline{D}$  is  connected.
Let $B$ be a subdomain of $D$ with smooth boundary $\partial B$.
 Then there exists a sequence
$(g_n)_{n\in \mathbb{N}}\subset L^2(\Omega)$, such that the corresponding solutions
$(u^{(g_n)})_{n\in \mathbb{N}}$ of (\ref{transm}) fulfill 
\begin{equation}
\label{localized_u}
\lim_{n\to \infty}\Vert u^{(g_n)}\Vert^2_{L^2(B)}=\infty,
\end{equation}
\begin{equation}
\label{localized_ubord}
\lim_{n\to \infty}\Vert  u^{(g_n)}\Vert^2_{H^1(D\setminus\overline{B})}=0,
\end{equation}
\begin{equation}
\label{localized_uB}
\lim_{n\to \infty}\Vert u^{(g_n)}\Vert^2_{L^2(\partial B)}=0,
\end{equation}
\begin{equation}
\label{localized_gradu}
\lim_{n\to \infty}\Vert \nabla u^{(g_n)}\Vert^2_{L^2(B)}=\infty.
\end{equation}
\end{theorem}
\begin{proof}
 This proof is based on the UCP for  Cauchy data. First, we define the virtual  measurement operators 
 $A_j$    ($j=1,2$) by
\[
A_1 :  L^2(B)\rightarrow  L^2(\partial\Omega), \quad  F\mapsto  v|_{\partial\Omega},
\]
 where $v\in H^1(\Omega)$ solves 
 \begin{equation}\label{dual_A} 
 \int_\Omega\sigma\nabla v\cdot\nabla w\,dx+\int_\Omega qvw\,dx=\int_{B}  Fw\,dx \quad \text{ for all } w\in H^1(\Omega), 
 \end{equation}
\[
A_2 :  H^1( D\setminus\overline{B})^{\prime} \rightarrow  L^2(\partial\Omega), \quad G\mapsto  v|_{\partial\Omega},
\]
 where $v\in H^1(\Omega)$ solves 
 \begin{equation}\label{dual_B}
 \int_\Omega\sigma\nabla v\cdot\nabla w\,dx+\int_\Omega qvw\,dx=\langle G, w\rangle_{D\setminus\overline{B}} \quad \text{ for all } w\in H^1(\Omega).
 \end{equation}
 Here $ \langle . , .\rangle_{D\setminus\overline{B}}$ denotes the dual pairing on $ H^1( D\setminus\overline{B})^{\prime}\times H^1( D\setminus\overline{B})$.
\noindent
First, we show that the dual operators $A'_1$ and $A'_2$  are given by
 \[
 \eqalign{
&A'_1 : L^2(\partial\Omega)  \rightarrow  L^2(B):  g \mapsto A'_1g= u|_{B}, \cr
&A'_2 : L^2(\partial\Omega)  \rightarrow   H^1( D\setminus\overline{B}): g\mapsto A'_2 g= u|_{D\setminus\overline{B}}.
}
\]
Let $F\in L^2(\Omega)$, $g\in L^2(\partial\Omega)$, $u$, $v\in H^1(\Omega)$ solve (\ref{transm}) and (\ref{dual_A}), respectively. Then,
\[
\int_{\Omega}F A^{\prime}_1 g\,dx 
= \int_{\partial\Omega}gA_1 F\,ds
=\int_{\Omega}\sigma\nabla v\cdot\nabla u\,dx+ \int_\Omega  q v u\,dx
=\int_{B} F u\,dx.
\]
 Let $G\in H^1(\Omega)$, $g\in L^2(\partial\Omega)$,  $u$, $v\in H^1(\Omega)$ solve (\ref{transm}) and (\ref{dual_B}), respectively. Then,
\[
\int_{\Omega}G  A^{\prime}_2 g \,dx
 = \int_{\partial\Omega}g A_2 G\,ds
 =\int_{\Omega}\sigma \nabla v \cdot \nabla u\,dx+\int_\Omega q v u\,dx
 =\langle G, u\rangle_{D\setminus\overline{B}}.
\]
\noindent
 Next, we will prove that
\[
\mathcal{R}(A_1)\cap  \mathcal{R}(A_2)=\lbrace 0\rbrace \quad \mathrm{and}\quad \mathcal{R}(A_1)\neq \{0\} .\label{range_A_1_B_2}
\]
Let  $\varphi \in \mathcal{R}(A_1)  \cap  \mathcal{R}(A_2)$.  Then there exist  $v_1, v_2
\in H^1(\Omega)$  such  that $ v_1|_{\partial\Omega}= v_2|_{\partial\Omega} =\varphi$,
and 
\[
\int_\Omega\sigma\nabla v_j\cdot \nabla w\,dx+ \int_\Omega q v_j w\,dx=0,
\]
for all $w\in H^1(\Omega)$  with  $\textrm{supp}(w)\subset \overline{\Omega}\setminus\overline{D}$.  
\noindent
Hence,
 \[
\mathrm{div}(\sigma\ \nabla v_j)+qv_j = 0\quad \text{ in }\Omega\setminus\overline{D},
\]
and $(\sigma\partial_\nu v_1)|_{\partial\Omega}=(\sigma\partial_n v_2)|_{\partial\Omega}=0$. 
The unique continuation principle for Cauchy data  yields that $v_1 = v_2$ in $\Omega\setminus\overline{D}$. Hence  
 $v:=  v_1\chi_{D\setminus\overline{B}}+v_2\chi_{\Omega\setminus (D\setminus\overline{B})}\in  H^1(\Omega)$ and  satisfies  
 \[
\left\{
\eqalign{
&\mathrm{div}(\sigma  \nabla v)+qv  = 0\quad \textrm{ in }\Omega,\cr
&\sigma\partial_\nu  v\  = 0\quad \textrm{ on }\partial\Omega.
}
\right.
\]
It follows  that $v=0$ and thus   $\varphi=v|_{\partial\Omega}=0$, and consequently \mbox{$\mathcal{R}(A_1)  \cap  \mathcal{R}(A_2)=\left\{0\right\}$}.
\\
\\
Next, we  will prove that  $\mathcal{R}(A_1)\neq \left\{0\right\}$.  We first prove the injectivity of the dual  operator $A_1^{\prime}$. Let  $g\in L^2(\partial\Omega)$  be such  that 
$A_1^{\prime}g= u\vert_D=0$.  By the unique continuation principal, we conclude that $u=0$ in $\Omega$.  This means that $g=\sigma\partial_\nu u\vert_{\partial\Omega}=0$,
which proves that $A_1^{\prime}$  is injective.  Hence  $A_1$  has a dense range, i.e.,   $\overline{\mathcal{R}(A_1)}=L^2(\partial\Omega)$.

A fortiori  $\mathcal{R}(A_1)\neq \{0\}$,   which together with  \mbox{$\mathcal{R}(A_1)  \cap  \mathcal{R}(A_2)=\left\{0\right\}$},  
 implies  the range non inclusion $\mathcal{R}(A_1)\not\subseteq \mathcal{R}(A_2)$.
Using \cite[Corollary 2.6]{gebauer2008localized}, it follows that there exists a sequence
$(g_n)_{n\in \mathbb{N}}\subset L^2(\partial\Omega)$  such  that 
\[
\lim_{n\to \infty}\Vert A_1^\prime g_n\Vert^2_{L^2(B)}=\lim_{n\to \infty} \Vert u^{(g_n)}\Vert^2_{L^2(B)}=\infty,
\]
\noindent
and
\begin{equation}
\label{op2}
\lim_{n\to \infty}\Vert A_2^\prime g_n\Vert^2_{H^1(D\setminus\overline{B})}=\lim_{n\to \infty} \Vert u^{(g_n)}\Vert^2_{H^1(D\setminus\overline{B})}=0.
\end{equation}
i.e.   (\ref{localized_u})   and (\ref{localized_ubord})  hold.  Also  (\ref{localized_uB}),  holds from  (\ref{op2}).  Since
\[
\Vert u^{(g_n)}\Vert_{L^2(B)} \leq C\left( \Vert u^{(g_n)}\Vert_{L^2(\partial B)}+\Vert \nabla u^{(g_n)}\Vert_{L^2(B)}  \right),
\]
where  $C>0$  is a constant,  this  also imply (\ref{localized_gradu}).
\end{proof}
\noindent
 Let $\mathcal{G}$ be a finite dimensional subset of  piecewise analytic  functions.   We consider four constants  $0<c_1\leq c_2$ and  $0< c_3 \leq c_4$
  which are the lower and upper bounds of the   parameters  and define the set
\[
 \mathcal{G}_{[c_1,c_2]\times[c_3,c_4]}=\left\{ (\sigma, q)\in \mathcal{G}:  \quad  c_1\leq\sigma(x)\leq c_2,  \quad  c_3\leq q(x)\leq c_4 \quad \textrm{ for all } x\in \Omega \right\}.
\]
In the following main result of this paper, the domain $\Omega$, the finite-dimensional subset $\mathcal{G}
$ and the bounds $0< c_1\leq c_2$ and $0< c_3 \leq c_4$ are fixed, and the constant in the Lipschitz stability result
 will depend on them.
\begin{theorem}[Lipschitz stability]
\label{stability2}
 There exists a positive constant $C>0$ such that for all 
$(\sigma_1, q_1), (\sigma_2, q_2) \in \mathcal{G}_{[c_1,c_2]\times[c_3,c_4]}$ with either
\[
\eqalign{
(a)\quad  \sigma_1\leq \sigma_2 \,\,\textrm{and}\,\, q_1\leq q_2\quad \textrm{or}\cr
(b)\quad \sigma_1\geq \sigma_2 \,\,\textrm{and}\,\, q_1\geq q_2,
}
\]
\noindent
we have 
\begin{equation}
\label{stab-est}
\eqalign{
d_\Omega( (\sigma_1,q_1),(\sigma_2,q_2)):=&\mathrm{max}\left(\Vert\sigma_1-\sigma_2\Vert_{L^\infty(\Omega)},  \Vert q_1-q_2\Vert_{L^\infty(\Omega)} \right)\cr
&\leq   C\| \Lambda(\sigma, q_1)-\Lambda(\sigma_2,q_2) \|_*.
}
\end{equation}
Here  $\Vert.\Vert_*$ is the natural norm of $\Vert.\Vert_{\mathcal L(L^2(\partial\Omega))}$. 
\end{theorem}
\begin{proof}
 For the sake of brevity, we write  $\Vert.\Vert$  for  $\Vert.\Vert_{L^2(\partial\Omega)}$.
 We start with the reformulation of the right-hand side of estimate (\ref{stab-est}).
Since $\Lambda(\sigma_1,q_1)$ and $\Lambda(\sigma_2, q_2)$ are self-adjoint, we have that
\[
\eqalign{
 \lefteqn{\Vert\Lambda(\sigma_2,q_2)-\Lambda(\sigma_1,q_1)\Vert_*}\cr
&=  \sup_{\Vert g\Vert=1} \left| \langle g, \left(\Lambda(\sigma_2,q_2)-\Lambda(\sigma_1,q_1)\right) g\rangle\right|\cr
&=  \sup_{\Vert g\Vert=1} \max\left\{ \langle g, \left(\Lambda(\sigma_2,q_2)-\Lambda(\sigma_1,q_1)\right) g\rangle,
\langle g, \left(\Lambda(\sigma_1,q_1)-\Lambda(\sigma_2,q_2)\right) g\rangle
\right\}.
}
\]
\noindent
Next, we apply both inequalities in the monotonicity relation (\ref{eqmono1}) in Lemma \ref{mono2}
 in order to obtain lower bounds for the corresponding integrals. 
We thus obtain for all $g\in L^2(\partial\Omega)$
\noindent
\begin{equation}\label{estim_1}
 \langle g,\left(\Lambda(\sigma_2,q_2)-\Lambda(\sigma_1,q_1)\right) g\rangle \geq \int_\Omega(\sigma_1-\sigma_2)\vert\nabla u_{(\sigma_1,q_1)}^{g}\vert^2\,dx
+\int_\Omega(q_1-q_2) \vert u^{g}_{(\sigma_1,q_1)}\vert^2\,dx
\end{equation}
\noindent
and
\begin{equation}\label{estim_2}
 \langle g,\left(\Lambda(\sigma_1,q_1)-\Lambda(\sigma_2,q_2)\right) g\rangle \geq \int_\Omega(\sigma_2-\sigma_2)\vert\nabla u_{(\sigma_2,q_2)}^{g}\vert^2\,dx
+\int_\Omega(q_2-q_2) \vert u^{g}_{(\sigma_2,q_2)}\vert^2\,dx
\end{equation}
where $u_{\sigma_1,q_1}^{g},u_{\sigma_2,q_2}^{g}\in  H^1(\Omega)$ denote the solutions of (\ref{transm}) with Neumann data $g$ and parameters  $(\sigma_1, q_1)$ and $(\sigma_2,q_2)$, respectively.
 Based on the estimates (\ref{estim_1}) and (\ref{estim_2}), we obtain for $(\sigma_1,q_1)\neq (\sigma_2,q_2)$
\begin{equation}\label{estim_3}
 \eqalign{
 &\frac{\Vert \Lambda(\sigma_2,q_2)-\Lambda(\sigma_1,q_1) {\Vert_*}}{d_\Omega( (\sigma_1,q_1),(\sigma_2,q_2))}\cr
&\geq  
 \sup_{\Vert g\Vert=1}\Phi\left(g,\frac{\sigma_1-\sigma_2}{d_\Omega( (\sigma_1,q_1),(\sigma_2,q_2))},\frac{q_1-q_2}{(d_\Omega( (\sigma_1,q_1),(\sigma_2,q_2))},(\sigma_1,q_1),(\sigma_2, q_2)\right),
 }
\end{equation}
 and define for $g\in L^2(\partial\Omega)$, $(\zeta_1,\zeta_2)\in \mathcal{G}$, and $(\kappa_1,\tau_1),(\kappa_2,\tau_2)\in 
\mathcal{G}_{[c_1,c_2]\times[c_3,c_4]}$ the function $\Phi\left(g,(\zeta_1,\zeta_2),(\kappa_1,\tau_1),(\kappa_2,\tau_2)\right)$ by
\[
\eqalign{
& \Phi\left(g,(\zeta_1,\zeta_2),(\kappa_1,\tau_1),(\kappa_2,\tau_2)\right)\\
 &:=\max \left(\Psi\left(g,(\zeta_1,\zeta_2),(\kappa_1,\tau_1)\right), \Psi\left(g,(-\zeta_1,-\zeta_2),(\kappa_2,\tau_2)\right)      \right),
}
\]
with
\[
\Psi\left(g,(\beta,\gamma),(\kappa,\tau )\right):=
\int_{\Omega} \beta \vert\nabla u_{(\kappa,\tau)}^{g}\vert^2\,dx+
\int_\Omega \gamma\vert u^{g}_{(\kappa,\tau)}\vert^2\,dx.
\]
\noindent
We introduce the compact sets
\[
 \eqalign{
 \mathcal{K}_+&=\left\{ (\zeta_1,\zeta_2)\in \rm{span }\,\mathcal{G}:  \quad  \zeta_1,\zeta_2\geq 0 \quad \text{  and }\quad 
 \max\left(\Vert \zeta_1\Vert_{L^\infty(\Omega)}, \Vert \zeta_2\Vert_{L^\infty(\Omega)}\right)= 1 \right\},\cr
 \mathcal{K}_-&=\left\{ (\zeta_1,\zeta_2)\in \rm{span }\,\mathcal{G}:  \quad   \zeta_1, \zeta_2\leq 0\quad \text{  and }\quad 
 \max\left(\Vert \zeta_1\Vert_{L^\infty(\Omega)}, \Vert\zeta_2\Vert_{L^\infty(\Omega)}\right)= 1 \right\},
 }
\]
\noindent
and denote $\mathcal{K}:= \mathcal{K}_+\cup \mathcal{K}_-$.
Then using that either assumption (a) or assumption (b) is fulfilled, we can rewrite (\ref{estim_3}) as
\begin{equation}\label{eq:stability_infsup}
\eqalign{
 &\frac{\Vert \Lambda(\sigma_2,q_2)-\Lambda(\sigma_1,q_1) {\Vert_*}}{d_\Omega( (\sigma_1,q_1),(\sigma_2,q_2))}\\
&\geq\inf_{\substack{(\zeta_1, \zeta_2)\in \mathcal{K}\cr (\kappa_1,\tau_1),(\kappa_2,\tau_2)\in\mathcal{G}_{[c_1,c_2]\times[c_3,c_4]}}} \sup_{\| g \|=1} \Phi\left(g,(\zeta_1,\zeta_2),(\kappa_1,\tau_1),(\kappa_2,\tau_2)\right).
}
\end{equation}
The assertion of Theorem \ref{stability} follows if we can show that the right-hand side of (\ref{eq:stability_infsup})
is positive. Since $\Phi$ is continuous,  we can conclude that the function
\[
\left((\zeta_1,\zeta_2),(\kappa_1,\tau_1),(\kappa_2,\tau_2)\right)\mapsto \sup_{\| g \|=1} \Phi\left(g,(\zeta_1,\zeta_2),(\kappa_1,\tau_1),(\kappa_2,\tau_2)\right),
\]
is semi-lower continuous, so that it attains its minimum on  the compact set \\
$\mathcal{K}\times \mathcal{G}_{[c_1,c_2]\times[c_3,c_4]}  \times \mathcal{G}_{[c_1,c_2]\times[c_3,c_4]} $.
Hence, to prove Theorem \ref{stability}, it suffices to show that
\begin{equation}\label{estim_4}
\sup_{\| g \|=1} \Phi\left(g,(\zeta_1,\zeta_2),(\kappa_1,\tau_1),(\kappa_2,\tau_2)\right)>0, 
\end{equation}
\noindent
for all $
\left((\zeta_1,\zeta_2),(\kappa_1,\tau_1),(\kappa_2,\tau_2)\right)\in \mathcal{K}\times \mathcal{G}_{[c_1,c_2]\times[c_3,dc_4]}  \times \mathcal{G}_{[c_1,c_2]\times[c_3,c_4]}.$

 In order to prove that (\ref{estim_4}) holds true, let $\left((\zeta_1,\zeta_2),(\kappa_1,\tau_1),(\kappa_2,\tau_2)\right)\in \mathcal{K}\times \mathcal{G}_{[c_1,c_2]\times[c_3,c_4]}  \times \mathcal{G}_{[c_1,c_2]\times[c_3,c_4]}$.
\\
We first treat the case that $(\zeta_1,\zeta_2)\in \mathcal{K}_+ $. Then there exist  an open subset $\emptyset\neq B\subset \Omega$ and  a constant $0<\delta<1$,  such that either
\[
\eqalign{
&\text{(i)}\   \zeta_1|_{B}\geq \delta,  \text{ and } \zeta_2\geq 0, \textrm{ or }\cr
&\text{(ii)}\   \zeta_2|_{B}\geq \delta,  \text{ and } \zeta_1\geq 0.
}
\]
We use the localized potentials sequence in Theorem \ref{thm:locpot} to obtain  a  boundary load
$\tilde g\in L^2(\partial\Omega)$ with 
\begin{equation}\label{estim_loc_pot}
\int_{B}\vert u^{\tilde g}_{(\kappa_1,\tau_1)}\vert^2\,dx  \geq \frac{1}{\delta }\quad \textrm{ and } \quad 
\int_{B}  \vert\nabla u^{\tilde g}_{(\kappa_1,\tau_1)}\vert^2\,dx  \geq \frac{1}{\delta} \quad.
\end{equation}
In case (i),  this leads to
\[
 \eqalign{
&\Phi\left(\tilde g,(\zeta_1,\zeta_2),(\kappa_1,\tau_1),(\kappa_2,\tau_2)\right)
\geq \int_\Omega \zeta_1\vert\nabla u^{\tilde g}_{(\kappa_1,\tau_1)}\vert^2\,dx+
\int_\Omega\zeta_2 \vert u^{\tilde g}_{(\kappa_1,\tau_1)}\vert^2\, dx \cr
& \geq \int_{B} \zeta_1 \vert\nabla u^{\tilde g}_{(\kappa_1,\tau_1)}\vert^2\,dx
 \geq \delta\int_{B}  \vert\nabla u^{\tilde g}_{(\kappa_1,\tau_1)}\vert^2\,dx \geq 1,
 }
\]
and in case (ii), we  have
\[
 \eqalign{
&\Phi\left(\tilde g,(\zeta_1,\zeta_2),(\kappa_1,\tau_1),(\kappa_2,\tau_2)\right)
\geq \int_\Omega \zeta_1\vert\nabla u^{\tilde g}_{(\kappa_1,\tau_1)}\vert^2\,dx+
\int_\Omega\zeta_2 \vert u^{\tilde g}_{(\kappa_1,\tau_1)}\vert^2\, dx  \cr
&\geq   \int_{B} \zeta_2 \vert u^{\tilde g}_{(\kappa_1,\tau_1)}\vert^2\, dx
 \geq \delta \int_{B} \vert u^{\tilde g}_{(\kappa_1,\tau_1)}\vert^2\, dx  \geq 1.
 }
\]
Hence, in both  cases, 
\[
 \eqalign{
\sup_{\| g \|=1} \Phi(g,(\zeta_1,\zeta_2),(\kappa_1,\tau_1),(\kappa_2,\tau_2))
&\geq  \Phi\left(\frac{\tilde g}{\Vert \tilde g\Vert},(\zeta_1,\zeta_2),(\kappa_1,\tau_1),(\kappa_2,\tau_2)\right)\cr
&=\frac{1}{\Vert\tilde g\Vert^2} \Phi(\tilde g,(\zeta_1,\zeta_2),(\kappa_1,\tau_1),(\kappa_2,\tau_2))>0.  
}
\]
\noindent
For  $(\zeta_1,\zeta_2) \in  \mathcal{K}_-$, we can analogously use a localized potentials sequence for $(\kappa_2,\tau_2)$, and 
prove that 
  \[ \sup_{\| g \|=1} \Phi(g,(\zeta_1,\zeta_2),(\kappa_1,\tau_1),(\kappa_2,\tau_2))>0,\]
 and the proof of  Theorem \ref{stability} is completed.
\end{proof}
\begin{remark} All the results of section 3 and section 4  stay valid  when the\\
 Neumann-to-Dirichlet operator $\Lambda(\sigma,q)$
is extended to  $H^{-\frac{1}{2}}(\partial\Omega)\to H^{\frac{1}{2}}(\partial\Omega)$.   On these spaces, it is easily
shown  that $\Lambda(\sigma,q)$  is bijective, and its inverse is the Dirichlet-to-Neumann operator
$\Lambda_D(\sigma,q): f \to u^{(f)}_{\sigma,q}\vert_{\partial\Omega}$, where $u^{(f)}_{\sigma,q}$  solves 
\[
\label{probD}
\left\{
\eqalign{
&-\nabla\cdot(\sigma \nabla u^{(f)}_{\sigma,q})+qu^{(f)}_{\sigma,q} = 0\quad \text{ in }  \Omega,\\
&u^{(f)}_{\sigma,q} = f\quad \text{ on  }\partial \Omega.
}
\right.
\]

\end{remark}

\section{Numerical approach to solve the inverse problem}
In this section,  we are interested in the following inverse problem 
\begin{equation}
\label{invp_num}
\text{ \it  Find }  \sigma, q  \text{ \it knowing measurements}\,\,f_k=\Lambda(\sigma,q)g_k,
\,\,k=1,\ldots K,
 \end{equation}
 \noindent
where $f_k\in L^2(\partial\Omega)$ is a measurement of the  density of photons  corresponding to the
input flux $g_k$, and $K\in \mathbb{N}$ is the number of measurements. 
\\
\\
To solve the inverse problem (\ref{invp_num}) numerically,
we consider a minimization problem of  a Kohn-Vogelius type functional:
\begin{equation}\label{func-J}
\eqalign{
 \min_{(\sigma,q)\in \mathcal{G}_{[c_1,c_2]\times[c_3,c_4]}}J(\sigma,q)=  
&\sum_{k=1}^K\int_{\Omega}\left( \sigma \vert\nabla (u^{(g_k)}-u^{(f_k)})\vert^2
+q \vert u^{(g_k)}-u^{(f_k)} \vert^2\right)\,dx \\
&+\frac{\rho}{2}\int_\Omega (\sigma^2 +q^2)\,dx.
}
\end{equation}
\noindent
Here $u^{(g_k)}$  and $u^{(f_k)} $  solve  the following problems:
\begin{equation}
\label{prob1}
\left\{
\eqalign{
&-\nabla\cdot(\sigma \nabla u^{(g_k)})+qu^{(g_k)} = 0\quad \text{ in }  \Omega,\cr
&\sigma\partial_{\nu} u^{(g_k)}= g_k\quad \text{ on  }\partial \Omega,
}
\right.
\end{equation}
\begin{equation}
\label{prob2}
\left\{
\eqalign{
&-\nabla\cdot(\sigma \nabla u^{(f_k)})+qu^{(f_k)} = 0\quad \text{ in }  \Omega,\\
&u^{(f_k)} = f_k\quad \text{ on  }\partial \Omega.
}
\right.
\end{equation}
When dealing with reconstruction of the absorption parameter $q$ where $\sigma$  is assumed to be known,
the minimization problem (\ref{func-J})  is reduced to 
\begin{equation}\label{min2}
\eqalign{
 \min_{q \in \mathcal{F}_{[a,b]}}\mathcal{J}(q)=  
\sum_{k=1}^K\int_{\Omega}\left( \sigma |\nabla (u^{(g_k)}-u^{(f_k)})|^2
+q |u^{(g_k)}-u^{(f_k)}|^2\right)\,dx +\frac{\rho}{2}\int_\Omega  q^2\,dx.
}
\end{equation}
\begin{theorem}\label{diffr_J}
 The functional $J:L^\infty_+(\Omega)^2 \to \mathbb{R}$, defined in (\ref{func-J})
 is Fr\'echet differentiable, and  its Fr\'echet derivative  at  
 $(\sigma,q)\in  L^\infty_+(\Omega)^2$ in the  
 direction  $(\tilde \sigma,  \tilde q) \in L^\infty_+(\Omega)^2$ is given by
 \begin{equation}
 \label{DJ}
   \eqalign{
  J^\prime\left(\sigma, q \right)(\tilde \sigma,\tilde q)
 =&\sum_{k=1}^K\int_\Omega\left(  \tilde{\sigma}\left( \vert\nabla u^{(f_k)}\vert^2 - \vert \nabla 
 u^{(g_k)}\vert^2 \right)
 +\tilde{q}\left( (u^{(f_k)})^2 - (u^{(g_k)})^2 \right)\right)\,dx\cr
 &+\rho\int_\Omega \left(\sigma\tilde\sigma+q\tilde q\right)\,dx.
 }
 \end{equation}
\end{theorem}
We need the following lemma  to prove Theorem \ref{diffr_J}.
\begin{lemma}
\label{der_op}
The non-linear operator 
\[
 \Lambda(\sigma, q): L^\infty_+(\Omega)^2\to \mathcal{L}(L^2(\partial\Omega)),  \quad (\sigma,q)\to \Lambda(\sigma,q)
 \]
is   Fr\'echet differentiable and its derivative
\[ 
\Lambda^\prime:  L^\infty_+(\Omega)^2\to  \mathcal{L}(L^\infty(\Omega)^2, \mathcal{L}(L^2(\partial\Omega))
\]
is given by the bilinear form
\begin{equation}
\label{dervL}
 \int_{\partial\Omega}g(\Lambda^\prime(\sigma,q)(\delta_1,\delta_2)) h\,ds=-\int_\Omega \delta_1 \nabla u_{\sigma,q}^{(g)}\cdot \nabla u_{\sigma,q}^{(h)}\,dx
-\int_\Omega \delta_2 u_{\sigma,q}^{(g)} u_{\sigma,q}^{(h)}\,dx,
\end{equation}
for all $\sigma, q\in L^\infty_+(\Omega)$,  $\delta_1, \delta_2 \in L^\infty(\Omega)$,  $g,  h \in L^2(\partial\Omega)$  
 where  $u^{(g)}_{\sigma,q}\in H^1(\Omega)$  is solution  of the problem (\ref{transm}). 
\end{lemma}
\begin{proof}
It follows from the monotonicity relation (\ref{eqmono})  that for all  sufficiently small  $\delta_1, \delta_2 \in L^\infty(\Omega)$  such that 
$\sigma+\delta_1, q+\delta_2\in L^\infty_+(\Omega)$
\[
 \eqalign{
&\int_\Omega  (\delta_1 \vert \nabla u_{\sigma,q}^{(g)}\vert^2+\delta_2 (u_{\sigma,q}^{(g)})^2)\,dx \geq \int_{\partial \Omega} g \left(\Lambda(\sigma, q)-\Lambda(\sigma+\delta_1, q+\delta_2)\right) g\, ds \cr
&\geq \int_\Omega\left(\sigma-\frac{\sigma^2}{\sigma+\delta_1}  \right)\vert \nabla u^{(g)}_{\sigma,q}\vert^2\,dx
+\int_\Omega\left(q-\frac{q^2}{q+\delta_2}  \right)(u^{(g)}_{\sigma,q})^2\,dx.
}
\]
Thus 
\begin{equation}
 \eqalign{
&\norm{ \Lambda(\sigma, q)-\Lambda(\sigma+\delta_1, q+\delta_2)-\Lambda^\prime(\sigma,q)(\delta_1,\delta_2) }_{\mathcal{L}(L^2(\partial\Omega))}\cr
&=\sup_{g\in L^2(\partial\Omega)} \left| \int_{\partial \Omega} g \left(\Lambda(\sigma,q)-\Lambda(\sigma+\delta_1,q+\delta_2)
-\Lambda^\prime(\sigma,q)(\delta_1,\delta_2)\right)\,ds\right|\cr
&\leq  \int_\Omega \left( \left(\frac{\delta^2_1}{\sigma+\delta_1}\right) \vert \nabla u^{(g)}_{\sigma,q}\vert^2
+\left(\frac{\delta^2_2}{q+\delta_2}\right) (u^{(g)}_{\sigma,q})^2\right)\,dx= O\left(\norm{(\delta_1,\delta_2)}^2_{\infty}\right).
}
\end{equation}
This shows that $\Lambda$  is   Fr\'echet differentiable, and its derivative is given by (\ref{dervL}).
\end{proof}
\begin{proof}[Proof of Theorem \ref{diffr_J}]
From the definition of the functional $J$, and applying\\
 Green's formula once, we have 
\begin{equation}
 \eqalign{
 J(\sigma,q)&=  \sum_{k=1}^K\int_{\Omega} \sigma |\nabla u^{(g_k)}|^2\,dx
+\sum_{k=1}^K\int_{\Omega} q|u^{(g_k)}|^2\,dx
+ \sum_{k=1}^K\int_{\Omega} \sigma |\nabla u^{(f_k)}|^2\,dx\cr
&+\sum_{k=1}^K\int_{\Omega} q|u^{(f_k)}|^2\,dx
-2\sum_{k=1}^K\int_{\partial\Omega} g_kf_k\,ds
+\frac{\rho}{2}\int_\Omega (\sigma^2+q^2)\,dx\cr
&= \sum_{k=1}^K \langle g_k,\Lambda(\sigma,q)g_k \rangle+   \sum_{k=1}^K \langle \Lambda_D(\sigma,q)f_k, f_k \rangle\\
&-2\sum_{k=1}^K\int_{\partial\Omega} g_kf_k\,ds+\frac{\rho}{2}\int_\omega (\sigma^2+q^2)\,dx.
}
\end{equation}
\noindent
From Lemma \ref{der_op},  $\Lambda(\sigma,q)$ is Fr\'echet differentiable with
\[
 \langle g_k,\Lambda^\prime(\sigma,q)(\tilde \sigma,\tilde{q})g_k\rangle
 =-\int_{\Omega}\left(\tilde{\sigma}\vert \nabla u^{(g_k)}\vert^2 
 +\tilde{q}(u^{(g_k)})^2 \right)\,dx,
\]
and
\[
\eqalign{
 \langle (\Lambda_D(\sigma,q))^\prime(\tilde \sigma,\tilde{q})f_k,f_k\rangle&=
 \langle (\Lambda(\sigma,q)^{-1})^\prime(\tilde \sigma,\tilde{q})f_k,f_k\rangle\cr
 &= \int_{\Omega}\left(\tilde{\sigma}\vert \nabla u^{(f_k)}\vert^2 
 +\tilde{q}(u^{(f_k)})^2 \right)\,dx  .
}
\]
Since   $\int_{\partial\Omega} g_kf_k\,ds$ is constant and $(\sigma,q)\to \int_\Omega (\sigma^2+q^2)\,dx$
is Fr\'echet differentiable,  we conclude that  $J$ is Fr\'echet differentiable and its derivative is given by (\ref{DJ}).
\end{proof}
\begin{remark}
Using the same techniques, we can prove that  the functional $\mathcal{J}$ is Fr\'echet differentiable  and its 
derivative is given by:
\[
\mathcal{J}^{\prime}(q)\tilde{q}=\sum_{k=1}^K\int_\Omega
 \tilde{q}\left( (u^{(f_k)})^2 - (u^{(g_k)})^2 \right)\,dx
 +\rho\int_\Omega \tilde q q\,dx.
\]
\end{remark}
\section{Implementation details and numerical examples}
We provide in this section two  numerical examples that illustrate the performance of our numerical method.
In the first example, we reconstruct the spatial distribution of the absorption coefficient while keeping the
diffusion  coefficient fixed.  In the second example,  we show  that  both optical properties are reconstructed simultaneously.

 When dealing with reconstruction with noise data,  the choice of the 
regularization parameter $\rho$  in (\ref{invp_num})  is crucial.  Usually,
it is determined using a knowledge of the noise level by,  e.g., the discrepancy principle.
However, in practice,  the noise level may be unknown, rendering such rules inapplicable.
To overcome  this issue, we propose a heuristic choice rule based on the following balancing
principle \cite{clason2010semismooth}:  Choose $\rho$  such  that 
\begin{equation}
\label{balance}
 (\beta-1) \sum_{k=1}^K\int_{\Omega}\left( \sigma \vert\nabla (u^{(g_k)}-u^{(f_k)})\vert^2+q \vert u^{(g_k)}-u^{(g_k)}\vert^2\right)\,dx -
\frac{\rho}{2}\int_\Omega (\sigma^2 +q^2)\,dx=0.
\end{equation}
 The idea  behind  this  principle is to balance the data fitting term with the
penalty term and  the weight  $\beta>1$  controls the trade-offš between them. The  choice rule
 does not require  the  knowledge of the noise level, and has been successfully applied
to linear and non linear  inverse problems \cite{clason2010duality,clason2010semismooth,clason2012semismooth,clason2012fitting,ito2011regularization}.\\
\noindent
When dealing only  with the reconstruction of $q$, the balancing  equation (\ref{balance})  is reduced to 
\begin{equation}
\label{balance1}
 (\beta-1) \sum_{k=1}^K\int_{\Omega}\left( \sigma |\nabla (u^{(g_k)}-u^{(f_k)})|^2
+q |u^{(g_k)}-u^{(f_k)}|^2\right)\,dx -
\frac{\rho}{2}\int_\Omega q^2\,dx=0.
\end{equation}
For our problem,  we compute a solution $\rho^*$ to the balancing equation (\ref{balance}) or  (\ref{balance1}) by the fixed point algorithm proposed in \cite{clason2012semismooth,clason2012fitting}.\\
\noindent
We consider the following setup for our numerical examples: The domain  $\Omega$ under consideration is the two dimensional unit disk centered at the origin. We use a Delaunay triangular 
mesh and a standard finite element method with piecewise  finite elements to numerically compute the states for our problem. The  measurements  $ f_k$ are computed synthetically 
by solving the direct problem (\ref{transm}).  To simulate noisy  data, the measurements $f_k$ are corrupted by adding
a normal Gaussian noise with mean zero and standard deviation  $\epsilon\Vert f_k\Vert_\infty$ where $\epsilon$  is a parameter.
To  avoid the so called ''inverse crime'', the inverse problem  is solved using $1016$ elements,  while the data $f_k$ is 
computed  with   $4064$  elements.  For all the computations  we have used  Matlab R2018a.
\subsection{Example 1:  Reconstructing $q$}\label{examp1}
In the following numerical results,   the diffusion coefficient $\sigma$  is  assumed to be  known,  and  is  given by 
$
\sigma=1\chi_{\Omega\setminus \overline{\omega}}+2\chi_{\omega},
$
where $\omega$ is the disk of radius $1/2$ centered at the origin.
The exact absorption coefficient to be recovered is given  by 
\[
q^{\dagger}(x_1,x_2)=1+\cos(\pi x_1) \cos(\pi x_2) \chi_{(\Vert (x_1,x_2)\Vert_{\infty}< 0.5)}.
\]
We obtain measurements $f_k$ corresponding to the fluxes 
\[
 g_k=10+\sin(k\theta), \quad \theta\in [0,2\pi], \quad k = 1,\ldots 5. 
 \]
 and we reconstruct $q$ by minimizing the functional 
 \[
 \mathcal{J}(q)=  
\sum_{k=1}^5\int_{\Omega}\left( \sigma |\nabla (u^{(g_k)}-u^{(f_k)})|^2
+q |u^{(g_k)}-u^{(f_k)}|^2\right)\,dx +\frac{\rho}{2}\int_\Omega  q^2\,dx,
\]
 in the space of piecewise  constant functions on the FEM mesh.\\
 \noindent
Figure \ref{fig1}  shows the true and the  reconstructed absorption  images  with    noise  free synthetic  data and without regularization.  Figure \ref{fig2} shows   the  reconstructed absorption images   with respect to  different   initialization and noise levels. The quality of  the reconstruction is satisfactory  and depend on the initialization of the algorithm.

\begin{figure}[ht!] 
\begin{center}
\begin{tabular}{c c c}
 \includegraphics[height=0.2\textheight]{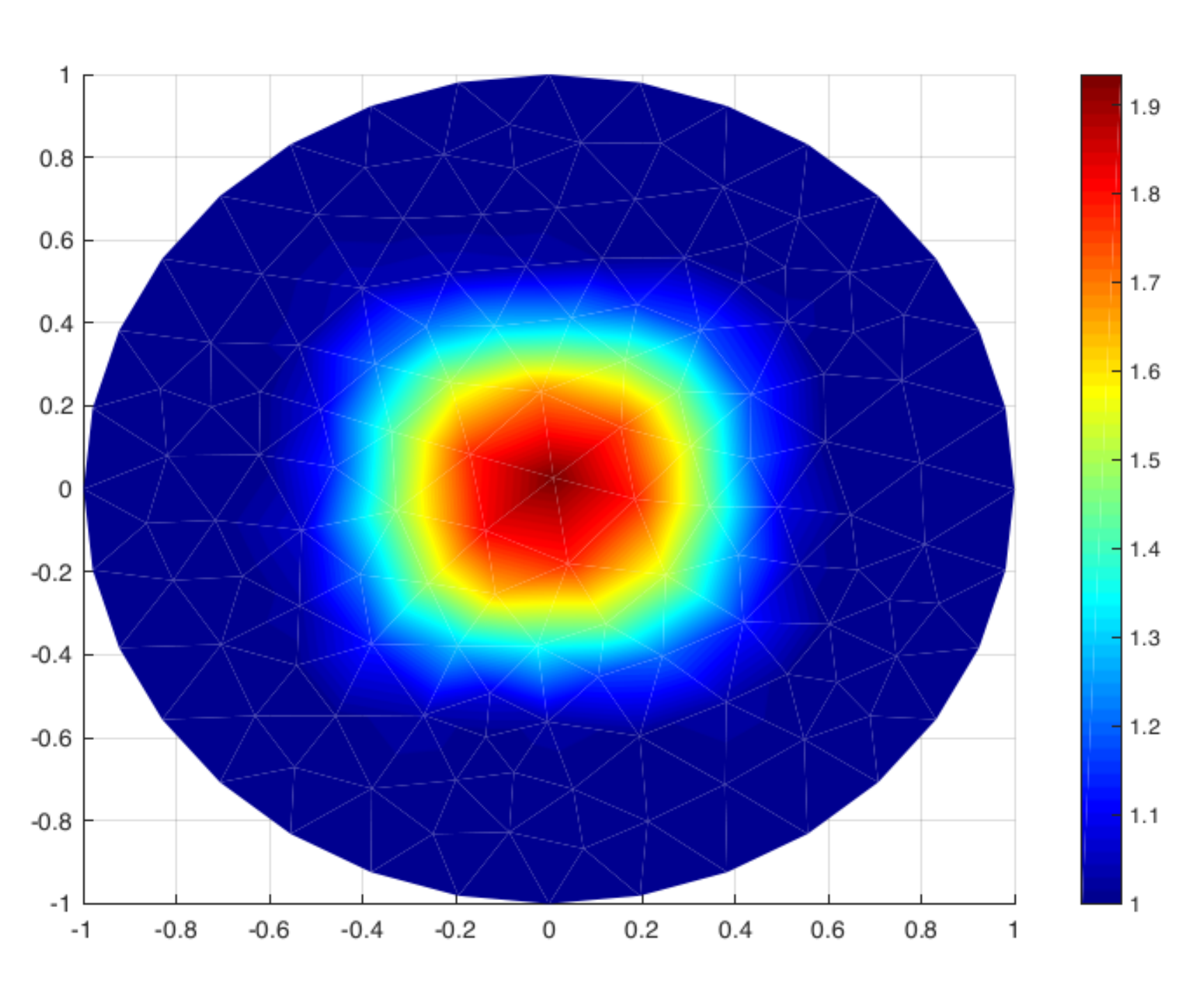}  &
 \includegraphics[height=0.2\textheight]{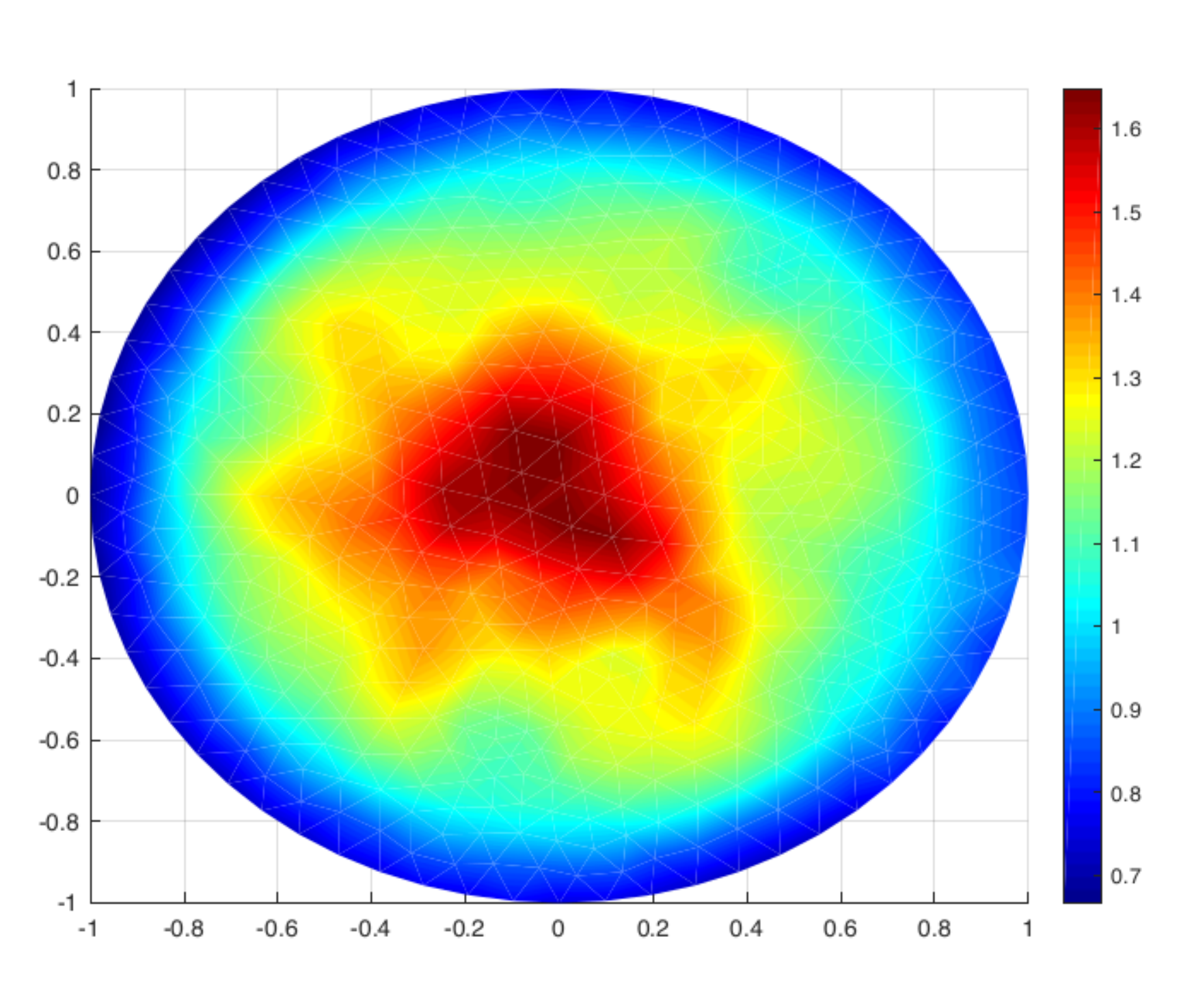} 
 \end{tabular}
 \end{center}
  \caption{On the left the true absorption  image  and  the  right the reconstructed absorption image   with $\epsilon=0$,  $\rho=0$  
  and  initialization  $ q(x_1,x_2)=1$.}
    \label{fig1}
 \end{figure}
\begin{figure}[ht!]
\begin{center}
\begin{tabular}{c c c}
 \includegraphics[height=0.2\textheight]{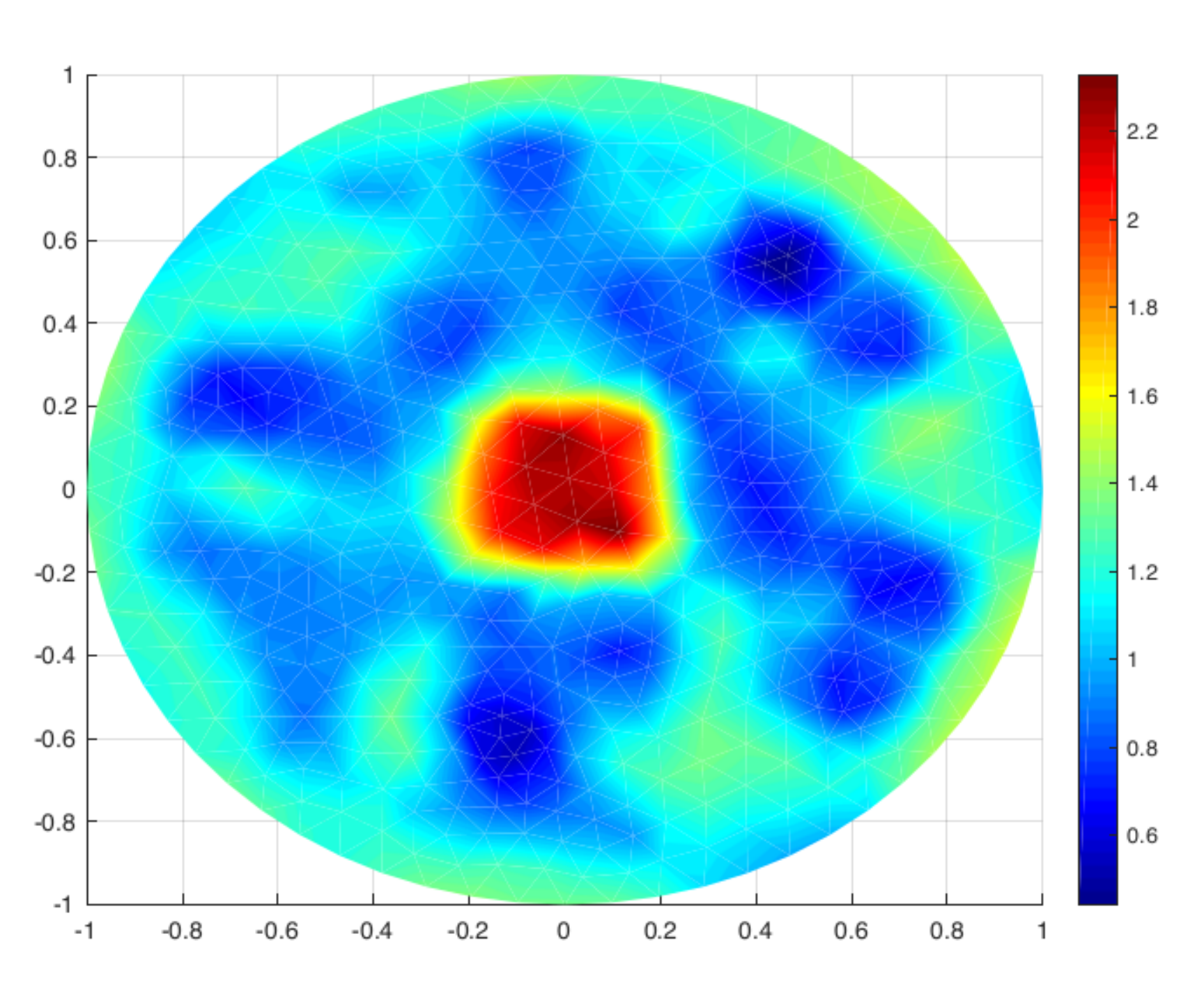}  &
 \includegraphics[height=0.2\textheight]{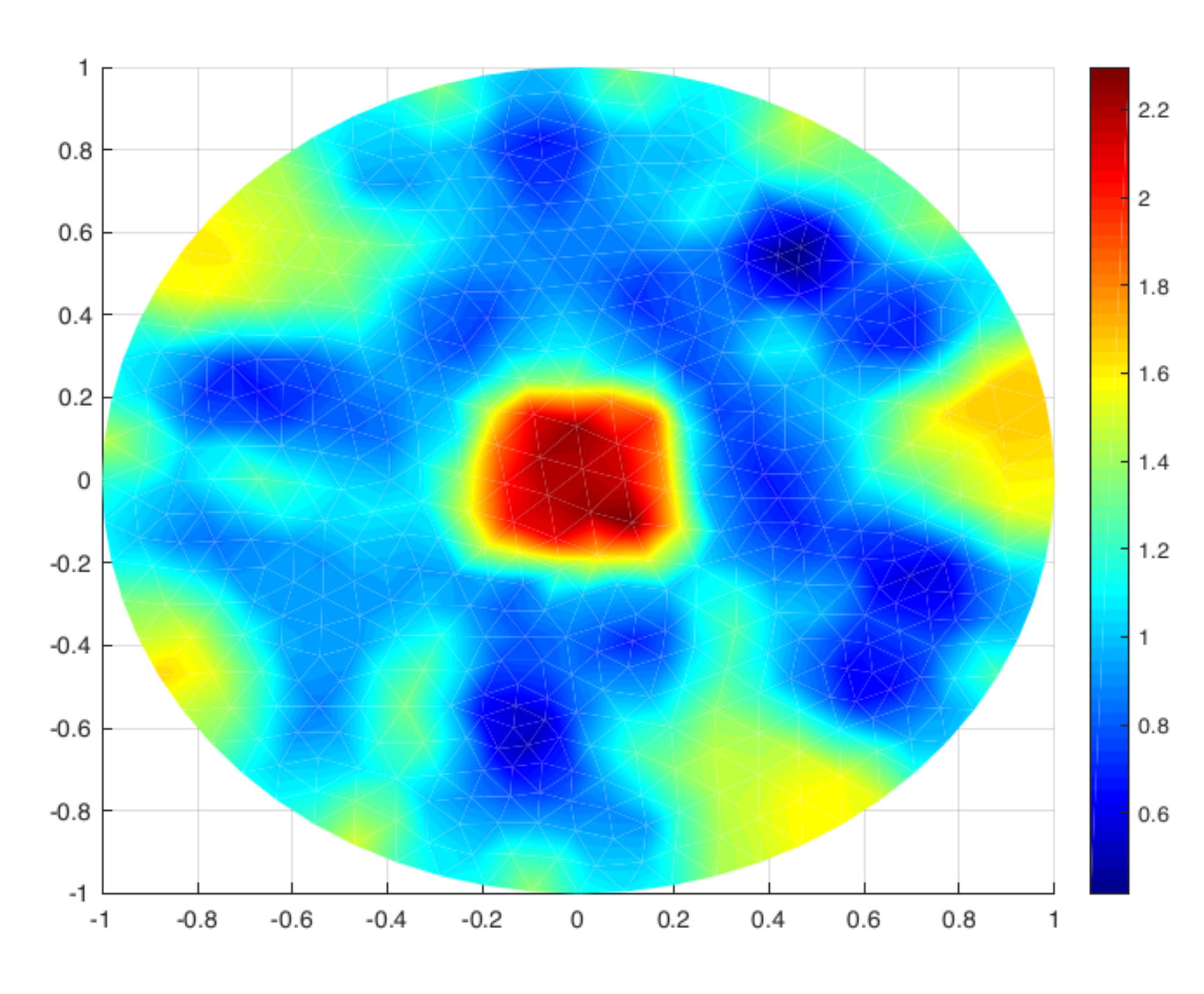}
 \end{tabular}
 \end{center}
  \caption{On the left   the reconstructed absorption image  with $\varepsilon=0$ and $\rho=0$.  On the  right  reconstructed absorption image  with 
  $\varepsilon= 0.05$ and $\rho=0.0000001672 $.  In both cases, the  initialization  is taken as $q(x_1,x_2)=(\vert x_1 \vert <0.2) (\vert x_2 \vert <0.2) $.}
  \label{fig2}
 \end{figure}
\subsection{Examples 2:  Reconstructing $\sigma$  and $q$  simultaneously}
In this example the  exact parameters to be recovered are given by
\[
 \sigma^{\dagger}(x)= 2\chi_{D_1}+3\chi_{D_2}+ 1\chi_{\Omega\setminus{\overline{D_1\cup D_2}}}, \quad 
q^{\dagger}(x)= 3\chi_{D_3}+4\chi_{D_4}+ 1\chi_{\Omega\setminus{\overline{D_3\cup D_4}}},
\]
where  $D_1$,  $ D_2$,  $ D_3$  and $ D_4$  are  given  by: 
\[
D_1=\left\{(x_1,x_2)\in \R^2: (x_1-0.5)^2+x_2^2<0.2^2\right\}, 
\]
\[
D_2=\left\{(x_1,x_2)\in \R^2: (x_1+0.5)^2+x_2^2<0.2^2\right\},
\]
\[
D_3=\left\{(x_1,x_2)\in \R^2: x_1^2+(x_2-0.5)^2<0.2^2\right\}, 
\]
\[
D_4=\left\{(x_1,x_2)\in \R^2: x_1^2+(x_2+0.5)^2<0.2^2\right\}.
\]
We use  measurements $f_k$ correspond to the fluxes $g_k(\theta)= \sin(k\theta),  \theta\in[0,2\pi], \quad k=1,\ldots 5$ and we reconstruct $\sigma, q$ by minimizing the function 
\[
 J(\sigma,q)=\sum_{k=1}^5\int_{\Omega}\left( \sigma |\nabla (u^{(g_k)}-u^{(f_k)}|^2
+q |u^{(g_k)}-u^{(f_k)}|^2\right)\,dx +\frac{\rho}{2}\int_\Omega (\sigma^2 +q^2)\,dx,
\]
 in the space of piecewise  constant functions on the FEM mesh. The initialization  is  given by 
 \[ 
 (\sigma(x),q(x))= 
 (1.1\chi_{D_1}+1.2\chi_{D_2}+ 1\chi_{\Omega\setminus{\overline{D_1\cup D_2}}},
 1.1\chi_{D_3}+1.2\chi_{D_4}+ 1\chi_{\Omega\setminus{\overline{D_3\cup D_4}}}).
 \]
 Figure \ref{fig3}. shows the true  diffusion  image  and  the reconstructed diffusion image  with  noise free  synthetic  data  and without regularization. Figure \ref{fig4}. depicts  the reconstructed  diffusion  images with   different noise  synthetic data and  regularization. Figure \ref{fig5}. depicts   the reconstructed  absorption image  with  noise free synthetic data and without regularization. Figure \ref{fig6}. shows  the reconstructed absorption image  with different noise synthetic data    and regularization.
 
In this example, the quality of reconstructions  is satisfactory  and the regularization technique  that we have  imposed here allows us  to estimate the optical  properties  in the presence of moderate noise  with  accuracy.   Let us mention that in  \cite{arridge1998gradient} the  authors  introduced a  gradient-based optimisation scheme  to  reconstruct  the optical  properties  without regularization 
of   the minimization problem.  A  crosstalk  problem    appeared in the  reconstruction  of the  profiles.  This is maybe due 
to   the non uniqueness of the inverse problem which is know to be severally ill-posed.   
 \begin{figure}[ht!]
\begin{center}
\begin{tabular}{c c c}
 \includegraphics[height=0.2\textheight]{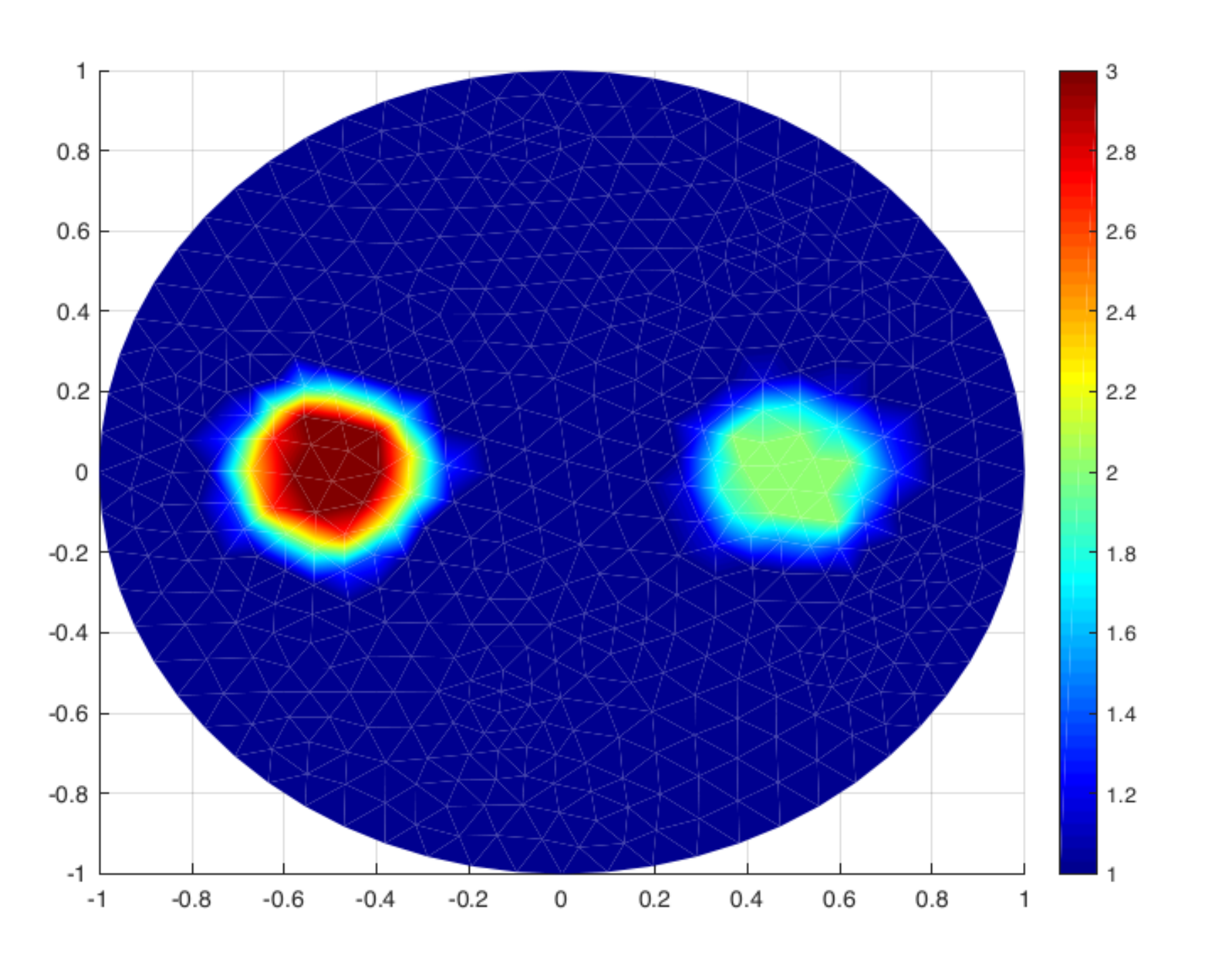}&
 \includegraphics[height=0.2\textheight]{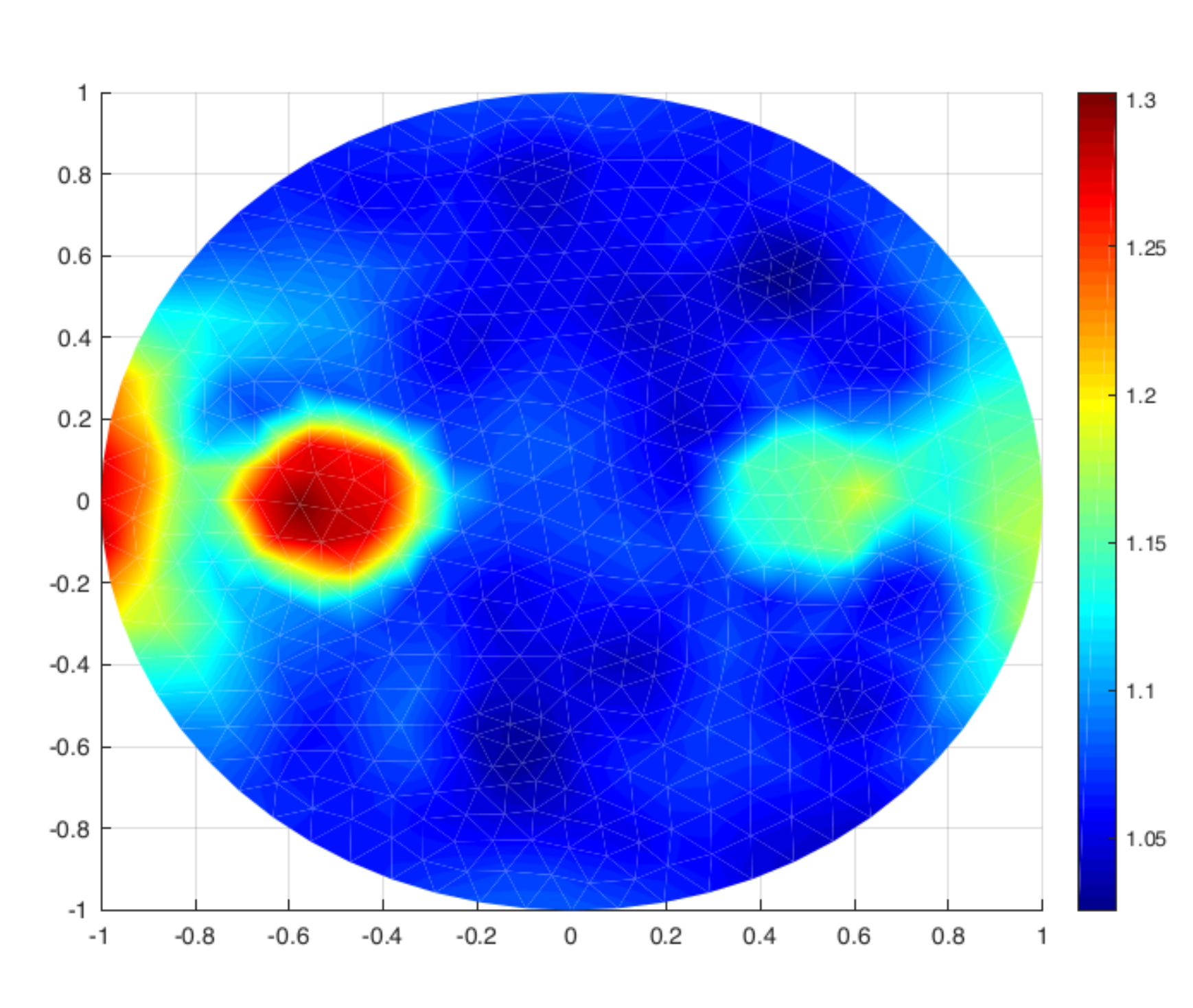} 
\end{tabular}
 \caption{On the left the true diffusion image and on the right the   reconstructed diffusion image  with
 $\varepsilon=0$,  $\rho=0$ .}
 \label{fig3}
\end{center}
\end{figure}
\begin{figure}[ht!]
\begin{center}
\begin{tabular}{c c c}
 \includegraphics[height=0.2\textheight]{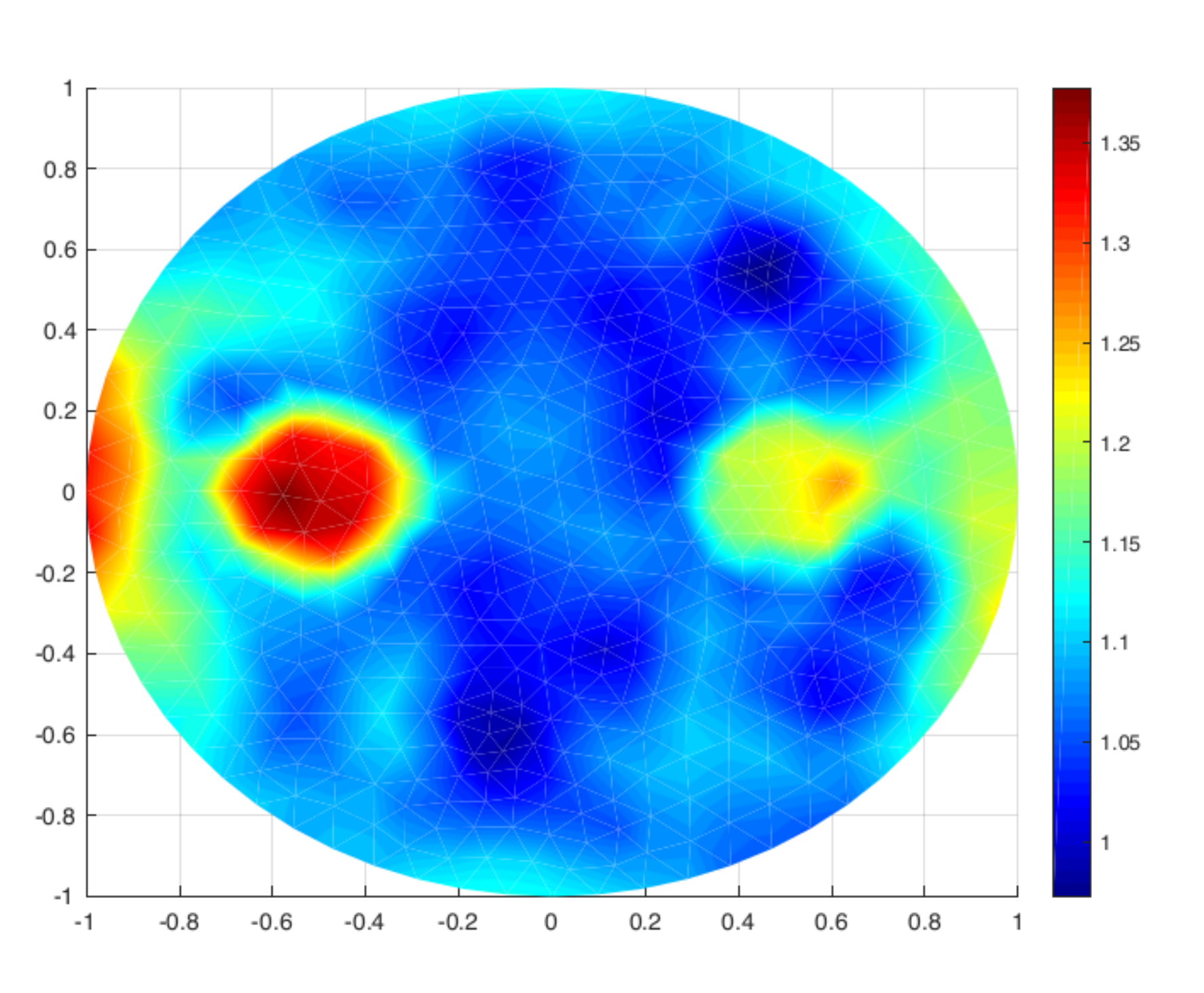}&
 \includegraphics[height=0.2\textheight]{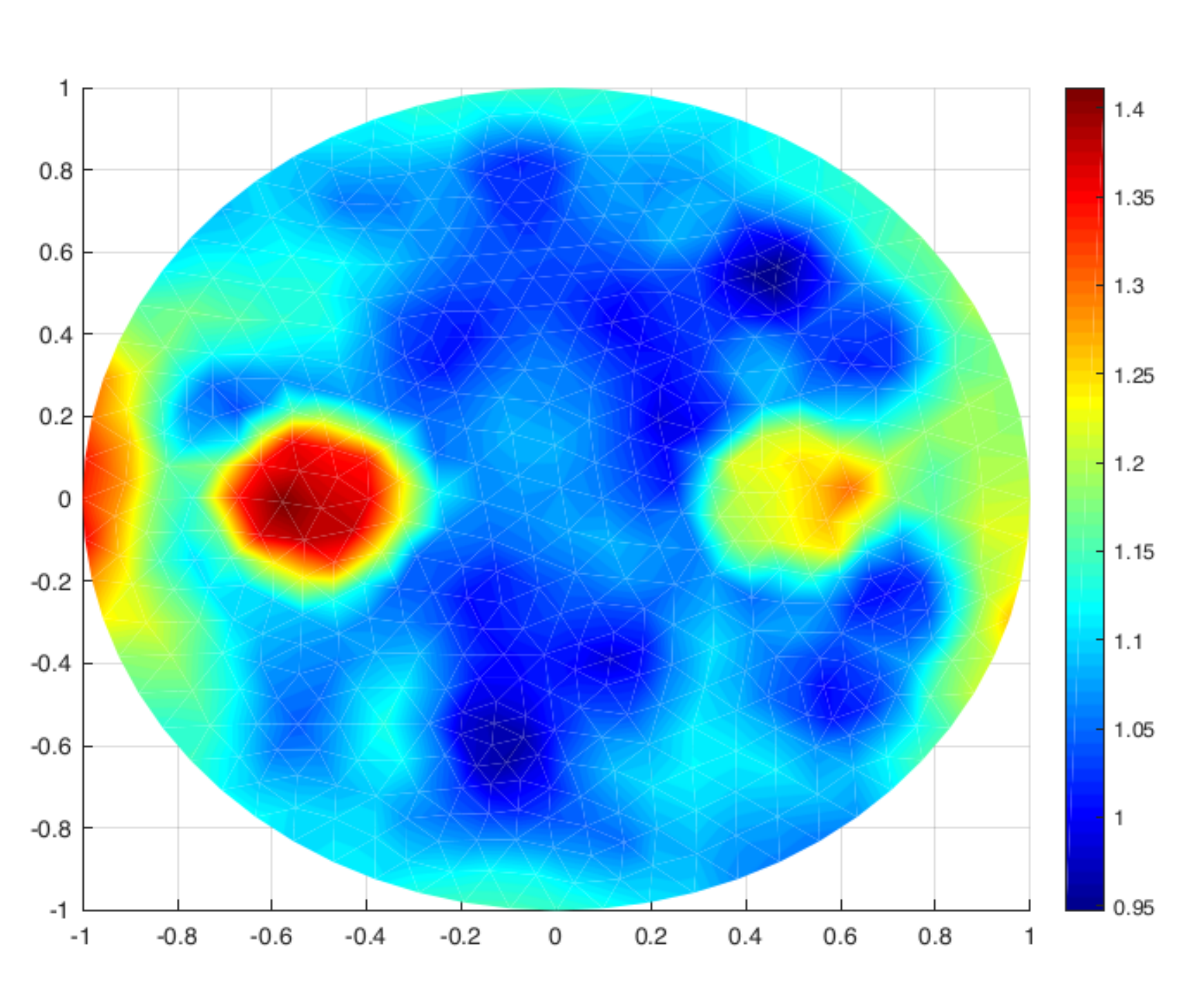} 
\end{tabular}
 \caption{On the left the reconstructed diffusion image with $\varepsilon=0.03$,  $\rho=1.674\times 10^{-6}$ and on the right the reconstructed diffusion  image 
 with $\varepsilon=0.05$  and $\rho=3.192\times 10^{-7}$.}
 \label{fig4}
\end{center}
\end{figure}
\begin{figure}[ht!]
\begin{center}
\begin{tabular}{c c c}
 \includegraphics[height=0.2\textheight]{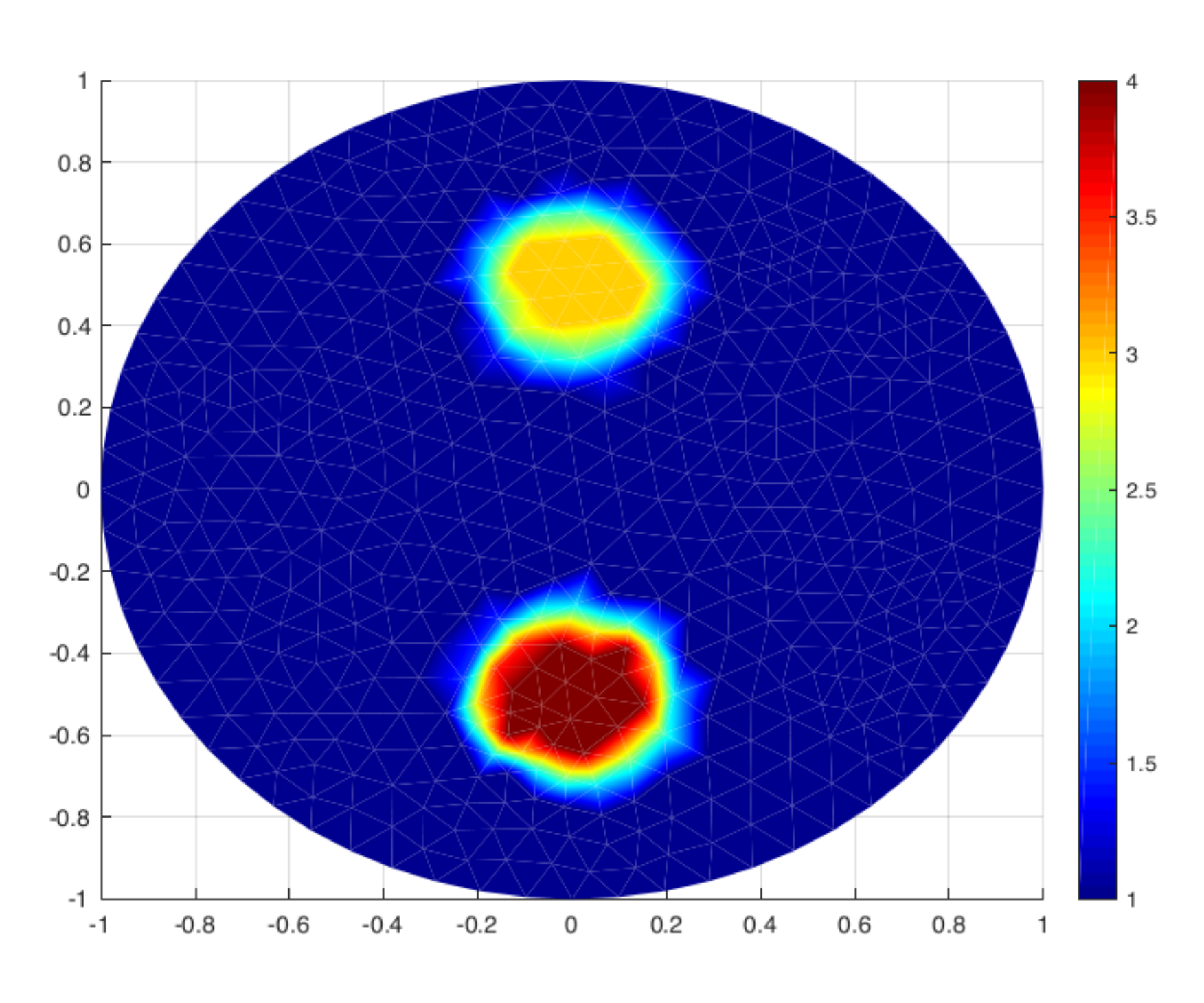}&
 \includegraphics[height=0.2\textheight]{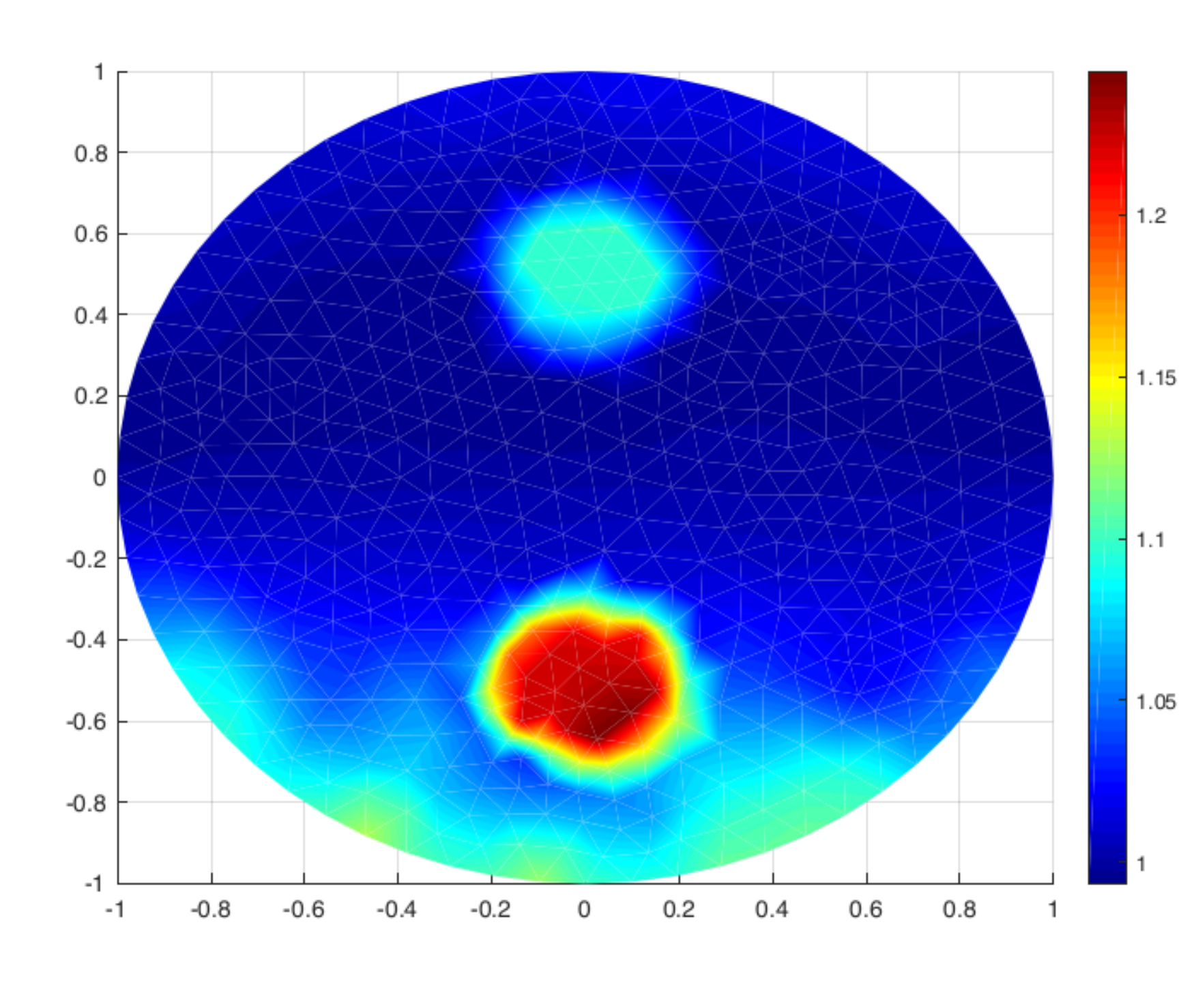} 
\end{tabular}
 \caption{On the left the  exact absorption  and on the right the reconstructed  absorption with $\varepsilon=0$,  and $\rho=0$.}
 \label{fig5}
\end{center}
\end{figure}
\begin{figure}[ht!]
\begin{center}
\begin{tabular}{c c c}
 \includegraphics[height=0.2\textheight]{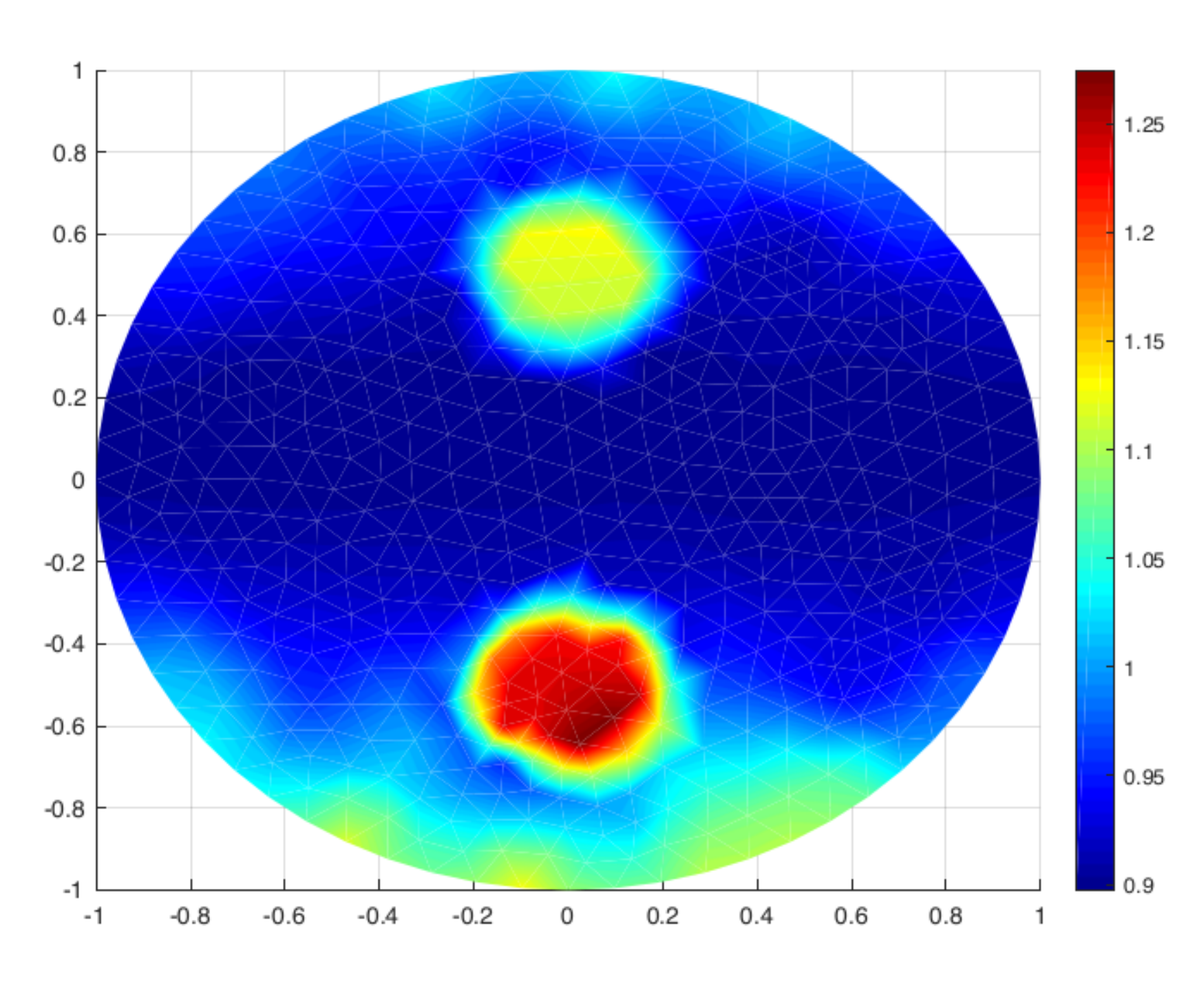}&
 \includegraphics[height=0.2\textheight]{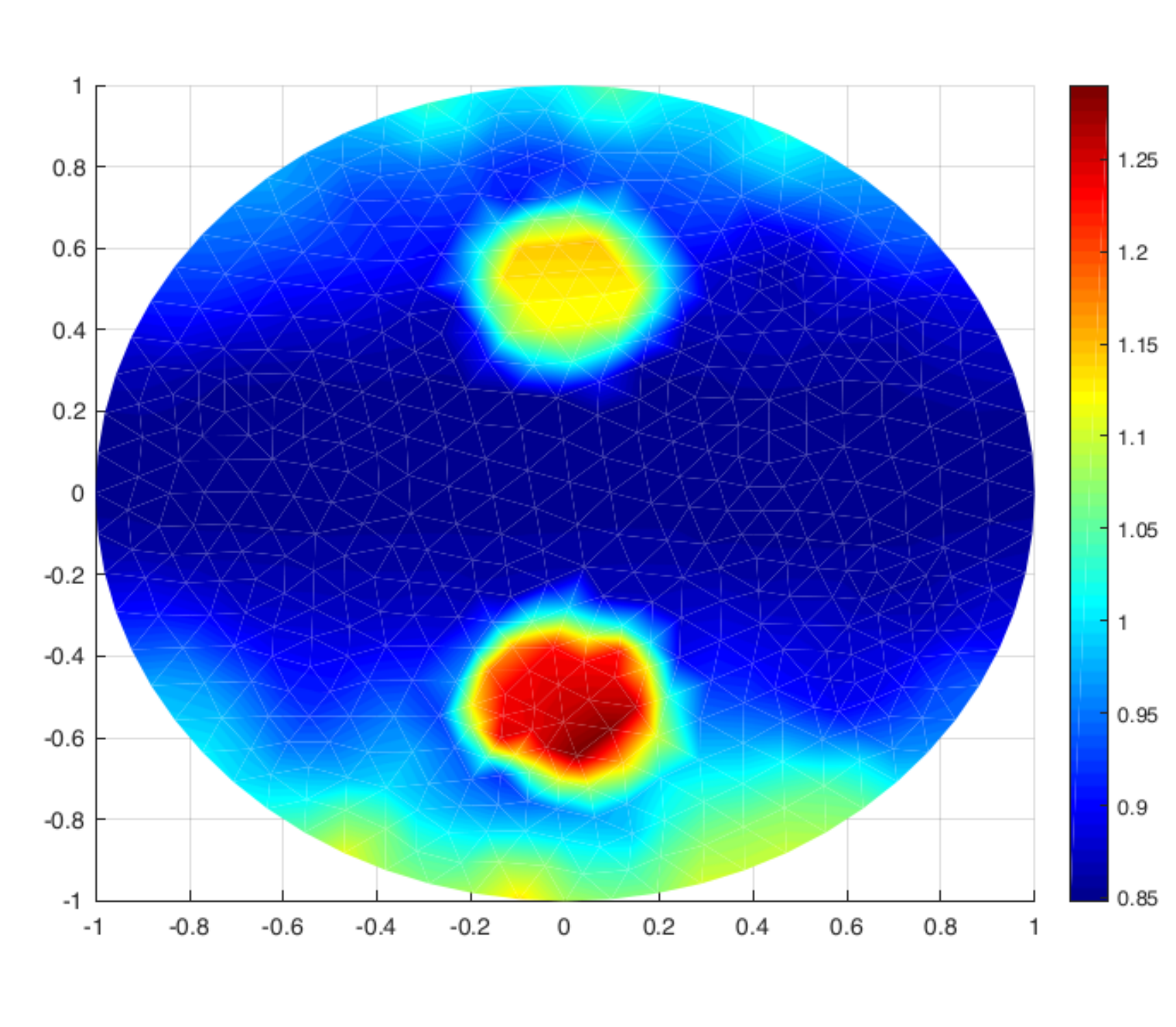} 
\end{tabular}
 \caption{On the left the reconstructed absorption with $\varepsilon=0.03$  and  $\rho= 5.82611\times 10^{-6}$,   and on the right the reconstructed absorption with $\varepsilon=0.05$  and  $\rho=1.35438\times 10^{-7}$.}
 \label{fig6}
\end{center}
\end{figure}
\section{Conclusion}
In this paper, we have shown a global uniqueness and  Lipschitz stability results  
 when  a-priori smoothness assumptions are imposed on the parameters ($\sigma$ piecewise constant and $q$ piecewise-analytic).  We have  also shown  for  a  given setting     that the Lipschitz  constant  can be computed by solving a finite numbers of  well posed PDEs.  The proofs  rely on the monotonicity of the NtD operator    combined with the techniques of localized potentials.   These  techniques seem  simple compared  to  the techniques of   Carleman estimates and complex geometrical  optics(CGO) used in  the litterature.   
 
 We have formulated the inverse problem   as a regularized problem    using    a Khon-Vogelius  functional. 
 In the inversion procedure, the forward model is discretized  using a finite  element  method. We solve the regularized  problem by using a  Quasi-Newton method with BFGS type updating rule for the Hessian matrix. Numerical reconstructions based on synthetic data provide results that are in agreement with the expected reconstructions  and no crosstalk between 
 the parameters is observed. 
 Let us mention  that  our   numerical method   depend strongly on the initialization, the measurements and the  mesh size.   
  When considering   the reconstruction of  $\sigma$  and $q$ simultaneously,   our algorithm  can't  reconstruct the jump  sets 
  of the parameters.   A shape optimization procedure may   be  used to  reconstruct the parameters and their jump sets  simultaneously.

\bibliographystyle{siam}
\bibliography{biblio}
\end{document}